\documentclass[a4paper,UKenglish,cleveref]{lipics-v2019}

\hypersetup{
    colorlinks=true,
    linkcolor=black,
    citecolor=black,
    filecolor=black,
    urlcolor=black,
}

\title{Domain Theory in Constructive and Predicative Univalent Foundations}

\titlerunning{Domain Theory in Constructive and Predicative UF}

\author{Tom de Jong}
{University of Birmingham, United Kingdom
\and \url{https://www.cs.bham.ac.uk/~txd880}}
{t.dejong@pgr.bham.ac.uk}
{https://orcid.org/0000-0003-1585-3172}
{}

\author{Mart{\'{i}}n H{\"{o}}tzel Escard{\'{o}}}
{University of Birmingham, United Kingdom
\and \url{https://www.cs.bham.ac.uk/~mhe}}
{m.escardo@cs.bham.ac.uk}
{https://orcid.org/0000-0002-4091-6334}
{}

\authorrunning{T. de {Jong} and M.\,H. {Escard{\'{o}}}}

\Copyright{Tom {de Jong} and Mart{\'{i}}n H. Escard{\'{o}}}

\ccsdesc[500]{Theory of computation~Constructive mathematics}
\ccsdesc[500]{Theory of computation~Type theory}

\keywords{domain theory, constructivity, predicativity, univalent foundations}

\category{} 

\relatedversion{A shorter version of this paper will appear in the proceedings
  of \emph{29th EACSL Annual Conference on Computer Science Logic (CSL 2021)},
  volume 183 of \emph{LIPIcs}.}

\nolinenumbers 

\hideLIPIcs  

\EventEditors{Christel Baier and Jean Goubault-Larrecq}
\EventNoEds{2}
\EventLongTitle{29th EACSL Annual Conference on Computer Science Logic (CSL 2021)}
\EventShortTitle{CSL 2021}
\EventAcronym{CSL}
\EventYear{2021}
\EventDate{January 25--28, 2021}
\EventLocation{Ljubljana, Slovenia (Virtual Conference)}
\EventLogo{}
\SeriesVolume{183}
\ArticleNo{19}

\usepackage{mathtools}
\usepackage{tikz-cd}

\theoremstyle{definition}
\newtheorem{construction}[theorem]{Construction}
\Crefname{construction}{Construction}{Constructions}

\DeclarePairedDelimiter{\paren}{(}{)}
\newcommand{\colonequiv}{\mathrel{\vcentcolon\mspace{-1.2mu}\equiv}}
\newcommand{\id}{\textup{\textsf{id}}}

\DeclareMathOperator{\Id}{\textup{\textsf{Id}}}
\DeclareMathOperator{\fst}{\textup{\textsf{pr}}_1}

\newcommand{\emptyt}{\mathbf{0}}
\newcommand{\unitt}{\mathbf{1}}
\newcommand{\Two}{\mathbf{2}}
\newcommand{\One}{\mathbf{1}}
\newcommand{\natt}{\mathbf{N}}
\DeclareMathOperator{\inl}{\textup{\textsf{inl}}}
\DeclareMathOperator{\inr}{\textup{\textsf{inr}}}
\DeclareMathOperator{\isprop}{\textup{\textsf{is-prop}}}
\newcommand{\hassize}{\mathrel{\textup{\textsf{has-size}}}}
\DeclarePairedDelimiter{\squash}{\|}{\|}
\DeclarePairedDelimiter{\tosquash}{|}{|}
\DeclareMathOperator{\Fin}{\textup{\textsf{Fin}}}
\newcommand{\locsmall}{\textup{\textsf{small}}}

\newcommand{\emptyisprop}{\textup{\textsf{\(\emptyt\)-is-prop}}}
\newcommand{\unitisprop}{\textup{\textsf{\(\unitt\)-is-prop}}}
\newcommand{\uniquefromempty}{\textup{\textsf{unique-from-\(\emptyt\)}}}

\newcommand{\appendixproof}[1]{}
\newcommand{\journalversion}[1]{}
\newcommand{\Omit}[1]{}
\newcommand{\U}{\mathcal U}
\newcommand{\V}{\mathcal V}
\newcommand{\W}{\mathcal W}
\newcommand{\T}{\mathcal T}

\newcommand{\DCPO}[3]{\mathcal {#1}\textup{-\textsf{DCPO}}_{\mathcal {#2,#3}}}
\newcommand{\DCPOnum}[3]{\U_{#1}\textup{-\textsf{DCPO}}_{\U_{#2},\U_{#3}}}
\newcommand{\order}{\sqsubseteq}
\newcommand{\below}{\order}
\DeclareMathOperator{\lifting}{\mathcal L}
\newcommand{\isdefined}{\mathpunct{\downarrow}}
\DeclareMathOperator{\isdirected}{\textup{\textsf{is-directed}}}
\DeclareMathOperator{\powerset}{\mathcal P}
\DeclareMathOperator{\List}{\textup{\textsf{List}}}

\DeclareMathOperator{\Idl}{\textup{\textsf{Idl}}}
\DeclareMathOperator{\isideal}{\textup{\textsf{is-ideal}}}
\DeclareMathOperator{\principalideal}{\downarrow}

\newcommand\twoheaddownarrow{\mathrel{\rotatebox[origin=c]{270}{$\twoheadrightarrow$}}}
\DeclareMathOperator{\ddarrow}{\twoheaddownarrow}

\newcommand{\dyadiccenter}{\textup{\textsf{center}}}
\DeclareMathOperator{\dyadicleft}{\textup{\textsf{left}}}
\DeclareMathOperator{\dyadicright}{\textup{\textsf{right}}}

\begin{document}

\maketitle

\begin{abstract}
  We develop domain theory in constructive univalent foundations
  without Voevodsky's resizing axioms. In previous work in this
  direction, we constructed the Scott model of PCF and proved its
  computational adequacy, based on directed complete posets
  (dcpos). Here we further consider algebraic and continuous dcpos,
  and construct Scott's \(D_\infty\) model of the untyped
  \(\lambda\)-calculus. A common approach to deal with size issues in
  a predicative foundation is to work with \emph{information systems}
  or \emph{abstract bases} or \emph{formal topologies} rather than
  dcpos, and \emph{approximable relations} rather than Scott
  continuous functions. Here we instead accept that dcpos may be large
  and work with type universes to account for this. For instance, in
  the Scott model of PCF, the dcpos have carriers in the second
  universe \(\U_1\) and suprema of directed families with indexing
  type in the first universe~\(\U_0\). Seeing a poset as a category in
  the usual way, we can say that these dcpos are large, but locally
  small, and have small filtered colimits.  In the case of algebraic
  dcpos, in order to deal with size issues, we proceed mimicking the
  definition of accessible category. With such a definition, our
  construction of Scott's~\(D_\infty\) again gives a large, locally
  small, algebraic dcpo with small directed suprema.
\end{abstract}

\section{Introduction}

In domain theory~\cite{AbramskyJung1994} one considers posets with suitable
completeness properties, possibly generated by certain elements called
\emph{compact}, or more generally generated by a certain \emph{way-below}
relation, giving rise to algebraic and continuous domains. As is well known,
domain theory has applications to programming language
semantics~\cite{Scott1993,Scott1972,Plotkin1977}, higher-type
computability~\cite{LongleyNormann2015}, topology, topological algebra and
more~\cite{GierzEtAl1980,GierzEtAl2003}.

In this work we explore the development of domain theory from the
univalent point of view~\cite{HoTTbook,UniMath}. This means that we
work with the stratification of types as singletons, propositions,
sets, 1-groupoids, etc., and that we work with univalence. At present,
higher inductive types other than propositional truncation are not
needed. Often the only consequences of univalence needed here are
functional and propositional extensionality. An exception is the
fundamental notion \emph{has size}: if we want to know that it is a
proposition, then univalence is necessary, but this knowledge is not
needed for our purposes (\cref{impredicativity}). Full details of our
univalent type theory are given in~\cref{foundations}.

Additionally, we work constructively (we
don't assume excluded middle or choice axioms) and predicatively (we don't assume
Voevodky's resizing principles~\cite{Voevodsky2011,Voevodsky2015,UniMath}, and
so, in particular, powersets are large). Most of the work presented here has
been formalized in the proof assistant
Agda~\cite{Agda,TypeTopology,TypeTopologyFork} (see \cref{conclusion}
for details). In our predicative setting, it is extremely important to
check universe levels carefully, and the use of a proof assistant such as Agda has been
invaluable for this purpose.

In previous work in this direction~\cite{deJong2019} (extended by
Brendan Hart~\cite{Hart2020}), we constructed the Scott model of PCF
and proved its computational adequacy, based on directed complete
posets (dcpos). Here we further consider algebraic and continuous
dcpos, and construct Scott's \(D_\infty\) model of the untyped
\(\lambda\)-calculus~\cite{Scott1972}.

A common approach to deal with size issues in a predicative foundation is to
work with \emph{information systems}~\cite{Scott1982}, \emph{abstract
  bases}~\cite{AbramskyJung1994} or \emph{formal
  topologies}~\cite{Sambin1987,CoquandEtAl2003} rather than dcpos, and
\emph{approximable relations} rather than (Scott) continuous functions. Here we
instead accept that dcpos may be large and work with type universes to account
for this. For instance, in our development of the Scott model of
PCF~\cite{Scott1993,Plotkin1977}, the dcpos have carriers in the second
universe~\(\U_1\) and suprema of directed families with indexing type in the
first universe~\(\U_0\). Seeing a poset as a category in the usual way, we can
say that these dcpos are large, but locally small, and have small filtered
colimits.  In the case of algebraic dcpos, in order to deal with size issues, we
proceed mimicking the definition of accessible
category~\cite{MakkaiPare1989}. With such a definition, our construction of
Scott's \(D_\infty\) again gives a large, locally small, algebraic dcpo with
small directed suprema.

  \paragraph*{Organization}
  \emph{\cref{foundations}}: Foundations.
  \emph{\cref{impredicativity}}: (Im)predicativity.
  \emph{\cref{basic:domain:theory}}: Basic domain theory, including
  directed complete posets, continuous functions, lifting,
\(\Omega\)-completeness, exponentials, powersets as dcpos.
\emph{\cref{D:infty}}: Limit and colimits of dcpos, Scott's
\(D_\infty\).
\emph{\cref{continuous:and:algebraic:dcpos}}: Way-below relation, bases, compact
element, continuous and algebraic dcpos, ideal completion, retracts,
examples.
\emph{\cref{conclusion}}: Conclusion and
future work.

\paragraph*{Related Work}

Domain theory has been studied predicatively in the setting of
\emph{formal topology}~\cite{Sambin1987,CoquandEtAl2003} in
\cite{SambinValentiniVirgili1996,Negri1998,Negri2002,MaiettiValentini2004}
and the more recent categorical paper \cite{Kawai2020}. In this
predicative setting, one avoids size issues by working with abstract
bases or formal topologies rather than dcpos, and approximable
relations rather than Scott continuous functions. Hedberg
\cite{Hedberg1996} presented these ideas in Martin-L\"of Type Theory
and formalized them in the proof assistant ALF. A~modern formalization
in Agda based on Hedberg's work was recently carried out in
Lidell's master thesis~\cite{Lidell2020}.

Our development differs from the above line of work in that it studies
dcpos directly and uses type universes to account for the fact that
dcpos may be large.
There are two Coq formalizations of domain theory in this direction
\cite{BentonKennedyVarming2009,Dockins2014}. Both formalizations study
\(\omega\)-chain complete preorders, work with setoids, and make use of Coq's
impredicative sort~\texttt{Prop}.
Our development avoids the use of setoids thanks to the adoption
of the univalent point of view. Moreover, we work predicatively and we
work with directed sets rather than \(\omega\)-chains, as we intend
our theory to be also applicable to topology and
algebra~\cite{GierzEtAl1980,GierzEtAl2003}.

There are also constructive accounts of domain theory aimed at program
extraction \cite{BauerKavkler2009,PattinsonMohammadian2020}.
Both \cite{BauerKavkler2009} and \cite{PattinsonMohammadian2020} study
\(\omega\)-chain complete posets (\(\omega\)-cpos) and define notions
of \(\omega\)-continuity for them. Interestingly, Bauer and Kavkler
\cite{BauerKavkler2009} note that there can only be non-trivial
examples of \(\omega\)-continuous \(\omega\)-cpos when Markov's
Principle holds \cite[Proposition 6.2]{BauerKavkler2009}. This leads
the authors of \cite{PattinsonMohammadian2020} to weaken the
definition of \(\omega\)-continuous \(\omega\)-cpo by using the
double negation of existential quantification in the definition of the
way-below relation \cite[Remark 3.2]{PattinsonMohammadian2020}. In
light of this, it is interesting to observe that when we study
directed complete posets (dcpos) rather than \(\omega\)-cpos, and
continuous dcpos rather than \(\omega\)-continuous \(\omega\)-cpos, we
can avoid Markov's Principle or a weakened notion of the way-below
relation to obtain non-trivial continuous dcpos (see for instance
\Cref{powerset-is-algebraic,lifting-is-algebraic,ideal-completion-of-dyadics}).

Another approach is the field of \emph{synthetic domain
  theory}~\cite{Rosolini1986,Rosolini1987,Hyland1991,Reus1999,ReusStreicher1999}. Although
the work in this area is constructive, it is still impredicative, based on topos
logic, but more importantly it has a focus different from that of regular domain
theory: the aim is to isolate a few basic axioms and find models in
(realizability) toposes where ``every object is a domain and every morphism is
continuous''. These models often validate additional axioms, such as Markov's
Principle and countable choice, and moreover falsify excluded middle.  Our
development has a different goal, namely to develop regular domain theory
constructively and predicatively, but in a foundation compatible with excluded
middle and choice, while not relying on them or Markov's Principle or countable
choice.

\section{Foundations} \label{foundations}
We work in intensional Martin-L\"of Type Theory with type formers
\(+\)~(binary sum), \(\Pi\)~(dependent products),
\(\Sigma\)~(dependent sum), \(\Id\) (identity type),
and inductive types, including~\(\emptyt\)~(empty type),
\(\unitt\)~(type with exactly one element \(\star : \unitt\)),
\(\natt\)~(natural numbers). Moreover, we have type universes (for
which we typically write \(\U\), \(\V\), \(\W\) or \(\T\)) with the
following closure conditions.
We assume a universe \(\mathcal U_0\) and two operations: for every universe
\(\U\) a successor universe \(\U^+\) with \(\U : \U^+\), and for every two
universes \(\U\) and \(\V\) another universe \(\U \sqcup \V\) such that for any
universe \(\U\), we have \(\U_0 \sqcup \U \equiv \U\) and
\(\U \sqcup \U^+ \equiv \U^+\). Moreover, \((-)\sqcup(-)\) is idempotent,
commutative, associative, and \((-)^+\) distributes over \((-)\sqcup(-)\). We
write \(\U_1 \colonequiv \U_0^+\), \(\U_2 \colonequiv \U_1^+, \dots\) and so on.
If \(X : \U\) and \(Y : \V\), then \({X + Y} : \U \sqcup \V\) and if \(X : \U\)
and \(Y : X \to \V\), then the types \(\Sigma_{x : X} Y(x)\) and
\(\Pi_{x : X} Y(x)\) live in the universe
$\U \sqcup \V\); finally, if~\(X : \U\) and \(x,y : X\), then
\(\Id_{X}(x,y) : \U\). The type of natural numbers \(\natt\) is assumed to be in
\(\U_0\) and we postulate that we have copies \(\emptyt_{\U}\) and
\(\unitt_{\U}\) in every universe \(\U\).  All our examples go through with just
two universes \(\U_0\) and \(\U_1\), but the theory is more easily developed in
a general setting.

In general we adopt the same conventions of~\cite{HoTTbook}.  In particular, we
simply write \(x=y\) for the identity type \(\Id_{X}(x,y)\) and use \(\equiv\)
for the judgemental equality, and for dependent functions
\(f,g : \Pi_{x : X}A(x)\), we write \(f \sim g\) for the pointwise equality
\(\Pi_{x : X} f(x) = g(x)\).

Within this type theory, we adopt the univalent point of
view~\cite{HoTTbook}.
A type \(X\) is a \emph{proposition} (or \emph{truth value} or
\emph{subsingleton}) if it has at most one element, i.e.\ the type
\(\isprop(X) \colonequiv \prod_{x,y : X} x = y\) is inhabited.
A major difference between univalent foundations and other foundational systems
is that we \emph{prove} that types are propositions or properties. For~instance,
we can show (using function extensionality) that the axioms of directed complete
poset form a proposition.
  A type \(X\) is a \emph{set} if any two elements can be identified in at most
  one way, i.e.\ the type \(\prod_{x,y : X} \isprop(x = y)\) is inhabited.

We will assume two extensionality principles:
\begin{enumerate}[(i)]
\item \emph{Propositional extensionality}: if \(P\) and \(Q\) are two
  propositions, then we postulate that \(P = Q\) exactly when both \(P \to Q\)
  and \(Q \to P\) are inhabited.
\item \emph{Function extensionality}: if \(f,g : \prod_{x : X}A(x)\) are two
  (dependent) functions, then we postulate that \(f = g\) exactly when
  \(f \sim g\).
\end{enumerate}
Function extensionality has the important consequence that the propositions form
an exponential ideal, i.e.\ if \(X\) is a type and \(Y : X \to \U\) is such that
every \(Y(x)\) is a proposition, then so is \(\Pi_{x : X}Y(x)\). In light of
this, universal quantification is given by \(\Pi\)-types in our type~theory.

In Martin-L\"of Type Theory, an element of
\(\prod_{x : X}\sum_{y : Y}\phi(x,y)\), by definition, gives us a function
\(f : X \to Y\) such that \(\prod_{x : X}\phi(x,f(x))\). In some cases, we wish
to express the weaker ``for every \(x : X\), there exists some \(y : Y\) such
that \(\phi(x,y)\)'' without necessarily having an assignment of \(x\)'s to
\(y\)'s. A good example of this is when we define directed families later (see
\cref{directed-family}). This is achieved through the propositional truncation.

Given a type \(X : \U\), we postulate that we have a proposition
\(\squash*{X} : \U\) with a function \({\tosquash{-} : X \to \squash*{X}}\) such
that for every proposition \(P : \V\) in any universe \(\V\), every function
\(f : X \to P\) factors (necessarily uniquely, by function extensionality)
through \(\tosquash{-}\).
Diagrammatically,
\begin{equation*}
  \begin{tikzcd}
    X \ar[dr, "\tosquash*{-}"'] \ar[rr, "f"] & & P \\
    & \squash*{X} \ar[ur, dashed]
  \end{tikzcd}
\end{equation*}
Existential quantification \(\exists_{x : X}Y(x)\) is given by
\(\squash*{\Sigma_{x : X}Y(x)}\). One should note that if we have
\(\exists_{x : X}Y(x)\) and we are trying to prove some proposition \(P\), then
we may assume that we have \(x : X\) and \(y : Y(x)\) when constructing our
inhabitant of \(P\). Similarly, we can define disjunction as
\(P \lor Q \colonequiv \squash*{P + Q}\).

\section{Impredicativity} \label{impredicativity}
We now explain what we mean by (im)predicativity in univalent foundations.

\begin{definition}[Has size, \href{https://www.cs.bham.ac.uk/~mhe/HoTT-UF-in-Agda-Lecture-Notes/index.html\#_has-size_}{\texttt{has-size}} \textnormal{in} \cite{Escardo2020}]
  A type \(X : \U\) is said to \emph{have size} \(\V\) for some universe \(\V\)
  when we have \(Y : \V\) that is equivalent to \(X\), i.e.\
  \(X \hassize \V \colonequiv \sum_{Y : \V} Y \simeq X\).
\end{definition}
Here, the symbol \(\simeq\) refers to Voevodsky's notion of
equivalence \cite{Escardo2020,HoTTbook}. Notice that the type
\(X \hassize \V\) is a proposition if and only if the univalence axiom
holds~\cite{Escardo2020}.

\begin{definition}[Type of propositions \(\Omega_{\U}\)]
  The type of propositions in a universe \(\U\) is
  \(\Omega_{\U} \colonequiv \sum_{P : \U} \isprop(P) : \U^+\).
\end{definition}

Observe that \(\Omega_{\U}\) itself lives in the successor universe
\(\U^+\). We often think of the types in some fixed universe \(\U\) as
\emph{small} and accordingly we say that \(\Omega_{\U}\) is
\emph{large}.
Similarly, the powerset of a type \(X : \U\) is large.  Given our
predicative setup, we must pay attention to universes when considering
powersets.

  \begin{definition}[\(\V \)-powerset \(\powerset_{\V }(X)\), \(\V \)-subsets]
    Let \(\V \) be a universe and \(X : \U \) type. We~define the
    \emph{\(\V \)-powerset} \(\powerset_{\V }(X)\) as
    \(X \to \Omega_{\V} : \V^+\sqcup \U \). Its elements are called
    \emph{\(\V \)-subsets} of \(X\).
  \end{definition}
  \begin{definition}[\(\in,\subseteq\)]
    Let \(x\) be an element of a type \(X\) and let \(A\) be an
    element of \(\powerset_{\V }(X)\). We write \(x \in A\) for the
    type \(\fst\paren*{A(x)}\).  The first projection \(\fst\) is
    needed because \(A(x)\), being of type \(\Omega_\V\), is a
    pair. Given two \(\V \)-subsets \(A\) and \(B\) of \(X\), we write
    \(A \subseteq B\) for
    \(\prod_{x : X}\paren*{x \in A \to x \in B}\).
  \end{definition}
Functional and propositional extensionality imply that \(A=B\) \(\iff\) \(A \subseteq B\) and \(B \subseteq A\).

  \begin{definition}[Total type \(\mathbb T(A)\)]
    Given a \(\V \)-subset \(A\) of a type \(X\), we write \(\mathbb T(A)\) for
    the \emph{total type} \(\sum_{x : X} x \in A\).
  \end{definition}

One could ask for a \emph{resizing axiom} asserting that
\(\Omega_{\U}\) has size \(\U\), which we call \emph{the propositional
  impredicativity of \(\U\)}. A closely related axiom is
\emph{propositional resizing}, which asserts that every proposition
\(P : \U^+\) has size \(\U\). Without the addition of such resizing
axioms, the type theory is said to be \emph{predicative}.  As an
example of the use of impredicativity in mathematics, we mention that
the powerset has unions of arbitrary subsets if and only if
propositional resizing holds
\cite[\href{https://www.cs.bham.ac.uk/~mhe/HoTT-UF-in-Agda-Lecture-Notes/HoTT-UF-Agda.html\#powerset-union-existence.existence-of-unions-gives-PR}{\texttt{existence-of-unions-gives-PR}}]{Escardo2020}.

We mention that the resizing axioms are actually theorems when classical logic
is assumed. This is because if \(P \lor \lnot P\) holds for every proposition in
\(P : \U\), then the only propositions (up to equivalence) are \(\emptyt_{\U}\)
and \(\unitt_{\U}\), which have equivalent copies in \(\U_0\), and
\(\Omega_{\U}\) is equivalent to a type \(\Two_{\U} : \U\) with exactly two
elements. The existence of a computational interpretation of propositional
impredicativity axioms for univalent foundations is an open problem,
however~\cite{Uemura2019,Swan2019}.

\section{Basic Domain Theory} \label{basic:domain:theory}
Our basic ingredient is the notion of \emph{directed complete poset}
(dcpo). In set-theoretic foundations, a dcpo can be defined to be a
poset that has least upper bounds of all directed subsets. A naive
translation of this to our foundation would be to proceed as
follows. Define a poset in a universe \(\U\) to be a type \(P:\U\)
with a reflexive, transitive and antisymmetric relation
\(-\below- : P \times P \to \U\).  According to the univalent point of
view, we also require that the type \(P\) is a \emph{set} and the
values \(p \below q\) of the order relation are
\emph{subsingletons}. Then we could say that the poset \((P,\below)\)
is \emph{directed complete} if every directed family \(I \to P\) with
indexing type \(I : \U\) has a least upper bound. The problem with
this definition is that there are no interesting examples in our
constructive and predicative setting. For instance, assume that the
poset~$\Two$ with two elements \(0\below 1\) is directed complete, and
consider a proposition~\(A:\U\) and the directed family
\(A +  \One \to \Two\) that maps the left component to~\(1\) and the
right component to~\(0\). By case analysis on its hypothetical
supremum, we conclude that the negation of \(A\) is decidable. This
amounts to weak excluded middle, which is known to be equivalent to De
Morgan's Law, and doesn't belong to the realm of constructive
mathematics. To try to get an example, we may move to the poset
\(\Omega_0\) of propositions in the universe \(\U_0\), ordered by
implication. This poset does have all suprema of families
\(I \to \Omega_0\) indexed by types \(I\) in the \emph{first universe}
\(\U_0\), given by existential quantification. But if we consider a
directed family \(I \to \Omega_0\) with \(I\) in the \emph{same
  universe} as \(\Omega_0\) lives, namely the \emph{second universe}
\(\U_1\), existential quantification gives a proposition in the
\emph{second universe} \(\U_1\) and so doesn't give its supremum. In this
example, we get a poset such that
\begin{enumerate}[(i)]
\item the carrier lives in the universe \(\U_1\),
\item the order has truth values in the universe \(\U_0\), and
\item suprema of directed families indexed by types in \(\U_0\) exist.
\end{enumerate}
Regarding a poset as a category in the usual way, we have a large, but
locally small, category with small filtered colimits (directed suprema). This
is typical of all the examples we have considered so far in practice,
such as the dcpos in the Scott model of PCF and Scott's \(D_\infty\)
model of the untyped \(\lambda\)-calculus. We may say that the
predicativity restriction increases the universe usage by one.
However, for the sake of generality, we formulate our definition of
dcpo with the following universe conventions:
\begin{enumerate}[(i)]
\item the carrier lives in a universe \(\U\),
\item the order has truth values in a universe \(\T\), and
\item suprema of directed families indexed by types in a universe \(\V\) exist.
\end{enumerate}
So our notion of dcpo has three universe parameters \(\U,\V,\T\). We
will say that the dcpo is \emph{locally small} when \(\T\) is not
necessarily the same as \(\V\), but the order has truth values of
size~\(\V\). Most of the time we mention \(\V\) explicitly and leave
\(\U\) and \(\T\) to be understood from the context.

\begin{definition}[Poset]
  A \emph{poset} \((P,\sqsubseteq)\) is a set \(P : \U \) together with a
  proposition-valued binary relation \({\sqsubseteq} : {P \to P \to \T} \)
  satisfying:
  \begin{enumerate}[(i)]
  \item \emph{reflexivity}: \(\prod_{p:P} p \sqsubseteq p\);
  \item \emph{antisymmetry}:
    \(\prod_{p,q:P} p \sqsubseteq q \to q \sqsubseteq p \to p = q\);
  \item \emph{transitivity}:
    \(\prod_{p,q,r:P} p \sqsubseteq q \to q \sqsubseteq r \to p \sqsubseteq r\).
  \end{enumerate}
\end{definition}

\begin{definition}[Directed family]
  \label{directed-family}
  Let \((P,\order)\) be a poset and \(I\) any type. A family
  \(\alpha : I \to P\) is \emph{directed} if it is inhabited (i.e.\
  \(\squash*{I}\) is pointed) and
  \(\Pi_{i,j : I}\exists_{k : I} \alpha_i \order \alpha_k \times \alpha_j \order
  \alpha_k\).
\end{definition}
\begin{definition}[\(\V \)-directed complete poset, \(\V \)-dcpo]
  Let \(\V \) be a type universe. A \emph{\(\V \)-directed
    complete poset} (or \emph{\(\V \)-dcpo}, for short) is a
  poset \(\paren*{P,\order}\) such that every directed family \(I \to P\) with
  \(I : \V \) has a supremum in \(P\).
\end{definition}
We will sometimes leave the universe \(\V \) implicit, and simply speak
of ``a dcpo''. On other occasions, we need to carefully keep track of universe
levels. To this end, we make the following definition.
\begin{definition}[\(\DCPO{V}{U}{T}\)]
  Let \(\V\), \(\U\) and \(\T \) be universes. We write
  \(\DCPO{V}{U}{T}\) for the type of \(\V \)-dcpos with carrier in
  \(\U \) and order taking values in \(\T \).
\end{definition}
\begin{definition}[Pointed dcpo]
  A dcpo \(D\) is \emph{pointed} if it has a least element, which we
  will denote by \(\bot_{D}\), or simply \(\bot\).
\end{definition}
\begin{definition}[Locally small]
  A \(\V \)-dcpo \(D\) is \emph{locally small} if we have
  \({\order_\locsmall} : D \to D \to \V\) such that
  \(\prod_{x,y : D} \paren*{x \order_\locsmall y} \simeq \paren*{x \order_{D}
    y}\).
\end{definition}

\begin{example}[Powersets as pointed dcpos]
  Powersets give examples of pointed dcpos.
  The subset inclusion \(\subseteq\) makes \(\powerset_{\V}(X)\) into a poset
  and given a (not necessarily directed) family
  \(A_{(-)} : I \to \powerset_{\V }(X)\) with \(I : \V \), we may consider its
  supremum in \(\powerset_{\V }(X)\) as given by
  \(\lambda x . \exists_{i : I}\, x \in A_i\). Note that
  \(\paren*{\exists_{i : I}\, x\in A_i} : {\V }\) for every \(x : X\), so this
  is well-defined. Finally, \(\powerset_{\V }\) has a least element, the empty
  set: \(\lambda x . \emptyt_{\V }\). Thus,
  \(\powerset_{\V}(X) : \DCPO{V}{V^+\sqcup U}{V\sqcup U}\). When
  \(\V \equiv \U\) (as in \cref{powerset-is-algebraic}), we get the simpler,
  locally small \(\powerset_{\U}(X) : \DCPO{U}{U^+}{U}\).\lipicsEnd
\end{example}
Fix two \(\V \)-dcpos \(D\) and \(E\).
\begin{definition}[Continuous function]
  A function \(f : D \to E\) is \emph{(Scott) continuous} if
  it preserves directed suprema, i.e.\ if \(I : \V \) and
  \(\alpha : I \to D\) is directed, then \(f\paren*{\bigsqcup \alpha}\)
  is the supremum in \(E\) of the family \(f \circ \alpha\).
\end{definition}
\begin{lemma}\label{continuous-implies-monotone}
  If \(f : D \to E\) is continuous, then it is monotone, i.e.\
  \(x \order_{D} y\) implies \(f(x) \order_{E} f(y)\).
\end{lemma}
  \begin{proof}
  Given \(x,y : D\) with \(x \order y\), consider the directed family
  \(\unitt + \unitt \to D\) defined as \(\inl(\star) \mapsto x\) and
  \(\inr(\star) \mapsto y\). Its supremum is \(y\) and \(f\) must preserve it,
  so \(f(x) \order f(y)\).
\end{proof}
\begin{lemma}
  If \(f : D \to E\) is continuous and
  \(\alpha : I \to D\) is directed, then so is \(f \circ \alpha\).
\end{lemma}
  \begin{proof}
  Using \cref{continuous-implies-monotone}.
\end{proof}
\begin{definition}[Strict function]
  Suppose that \(D\) and \(E\) are pointed. A continuous
  function \(f : D \to E\) is \emph{strict} if
  \(f\paren*{\bot_{D}} = \bot_{E}\).
\end{definition}
\subsection{Lifting}
\begin{construction}[\(\lifting_{\V }(X)\), \(\eta_X\),
  \textnormal{cf.}~\cite{deJong2019,
    EscardoKnapp2017}]\label{construction-lifting}
  Let \(X : \U \) be a set. For any universe \(\V \), we construct
  a pointed \(\V \)-dcpo
  \(\lifting_{\V }(X) : \DCPO{V}{V^+\sqcup U}{V^+\sqcup U}\), known as
  the \emph{lifting} of \(X\). Its carrier is given by the type
  \( \sum_{P : \V }\isprop(P) \times (P \to X) \) of \emph{partial elements} of~\(X\).

  Given a partial element \(\paren*{P,i,\varphi} : \lifting_{\V }(X)\), we write
  \((P,i,\varphi) \isdefined\) for \(P\) and say that the partial element is
  defined if \(P\) holds. Moreover, we often leave the second component
  implicit, writing \((P,\varphi)\) for \((P,i,\varphi)\).

  The order is given by
  \( l \sqsubseteq_{\lifting_{\V }(X)} m \colonequiv \paren*{{l\isdefined} \to
    {l = m}} \), and it has a least element given by \((\emptyt, \emptyisprop,\)
  \(\uniquefromempty)\) where \(\emptyisprop\) is a witness that the empty type
  is a proposition and \(\uniquefromempty\) is the unique map from the empty
  type.

  Given a directed family
  \(\paren*{Q_{(-)},\varphi_{(-)}} : I \to \lifting_{\V }(X)\), its
  supremum is given by \(\paren*{\exists_{i : I} Q_i , \psi}\), where \(\psi\)
  is such that
  \[
    \begin{tikzcd}
      \sum_{i : I} Q_i \ar[dr,"\tosquash{-}"'] \ar{rr}{(i, q) \mapsto \varphi_i(q)} & & D \\
      & \exists_{i : I}Q_i \ar[ur, "\psi"']
    \end{tikzcd}
  \]
  commutes. (This is possible, because the top map is weakly constant (i.e.\ any
  of its values are equal) and \(D\) is a set~\cite[Theorem~5.4]{KrausEtAl2017}.)

  Finally, we write \(\eta_X : X \to \lifting_{\V }(X)\) for the embedding
  \(x \mapsto \paren*{\unitt, \unitisprop, \lambda u.x} \). \lipicsEnd
\end{construction}
Note that we require \(X\) to be a set, so that \(\lifting_{\V }(X)\) is
a poset, rather than an \(\infty\)-category.
In~practice, we often have \(\V \equiv \U \) (see for instance
\cref{lifting-is-algebraic}, \cref{Scott's-example-using-self-exponentiation},
or the Scott model of PCF \cite{deJong2019}), but we develop the theory for the
more general case.
We can describe the order on \(\lifting_{\V }(X)\) more explicitly, as
follows.
\begin{lemma}\label{lifting-order-alt}
  If we have elements \((P,\varphi)\) and \((Q,\psi)\) of
  \(\lifting_{\V }(X)\), then \((P,\varphi) \order (Q,\psi)\) holds if
  and only if we have \(f : P \to Q\) such that
  \(\prod_{p : P}\varphi(p) = \psi(f(p))\).
\end{lemma}
Observe that this exhibits \(\lifting_{\V }(X)\) as locally small.
We will show that \(\lifting_{\V }(X)\) is the \emph{free} pointed
\(\V \)-dcpo on a set \(X\), but to do that, we first need a lemma.

\begin{lemma}\label{pointed-dcpos-have-subsingleton-joins}
  Let \(D\) be a pointed \(\V \)-dcpo. Then \(D\) has suprema of families
  indexed by propositions in \(\V \), i.e.\ if \(P : \V \) is a proposition,
  then any \(\alpha : P \to D\) has a supremum \(\bigvee \alpha\).

  Moreover, if \(E\) is another pointed \(\V \)-dcpo and \(f : D \to E\) is
  strict and continuous, then \(f\paren*{\bigvee \alpha}\) is the supremum of
  the family \(f \circ \alpha\).
\end{lemma}
\begin{proof}
  Let \(D\) be a pointed \(\V \)-dcpo, \(P : \V \) a
  proposition and \(\alpha : P \to D\) a function. Now define
  \(\beta : \unitt_{\V } + P \to D\) by
  \(\inl(\star) \mapsto \bot_{D}\) and \(\inr(p) \mapsto \alpha(p)
  \). Then, \(\beta\) is easily seen to be directed and so it has a well-defined
  supremum in \(D\), which is also the supremum of \(\alpha\). The
  second claim follows from the fact that \(\beta\) is directed, so continuous
  maps must preserve its supremum.
\end{proof}

\begin{lemma}\label{lifting-element-as-sup}
  Let \(X : \U \) be a set and let \((P,\varphi)\) be an arbitrary
  element of \(\lifting_{\V }(X)\). Then
  \((P,\varphi) = \bigvee_{p : P} \eta_X\paren*{\varphi(p)}\).
\end{lemma}

\begin{theorem}\label{lifting-is-free}
  The lifting \(\lifting_{\V }(X)\) gives the free pointed \(\V \)-dcpo on a set
  \(X\). Put precisely, if \(X : \U \) is a set, then for every pointed
  \(\V \)-dcpo \(D : \DCPO{V}{U'}{T'}\) and function \(f : X \to D\), there is a
  unique strict and continuous function
  \(\overline{f} : \lifting_{\V }(X) \to D\) such that
  \[
    \begin{tikzcd}
      X \ar[dr, "\eta_X"'] \ar[rr, "f"] & & D  \\
      & \lifting_{\V }(X) \ar[ur, dashed, "\overline{f}"']
    \end{tikzcd}
  \]
  commutes.
\end{theorem}
  \begin{proof}
  We define \(\overline f : \lifting_{\V }(X) \to D\) by
  \((P,\varphi) \mapsto \bigvee_{p : P} f(\varphi(p))\), which is well-defined
  by \cref{pointed-dcpos-have-subsingleton-joins} and easily seen to be
  continuous. For uniqueness, suppose that we have
  \({g : \lifting_{\V }(X) \to D}\) strict and continuous such that
  \(g \circ \eta_X = f\). Let \((P,\varphi)\) be an arbitrary element of
  \(\lifting_{\V }(X)\). Using \cref{lifting-element-as-sup}, we have:
  \begin{align*}
    g\paren*{P,\varphi}& = g\paren*{\bigvee_{p : P}\eta_X\paren*{\varphi(p)}} \\
    & = \bigvee_{p : P}g\paren*{\eta_X\paren*{\varphi(p)}}
    && \text{(by \cref{pointed-dcpos-have-subsingleton-joins} and continuity of \(g\))} \\
    & = \bigvee_{p : P} f\paren*{\phi(p)}
    && \text{(by assumption on \(g\))} \\
    & = \overline{f}(P,\varphi)
    && \text{(by definition)},
  \end{align*}
  as desired.
\end{proof}

There is yet another way in which the lifting is a free construction. What is
noteworthy about this is that freely adding subsingleton suprema automatically
gives all directed suprema.

\begin{definition}[\(\Omega_{\V }\)-complete]
  A poset \((P,\order)\) is \emph{\(\Omega_{\V }\)-complete} if it has
  suprema for all families indexed by a proposition in \(\V \).
\end{definition}
\begin{theorem}
  \label{lifting-is-free2}
  The lifting \(\lifting_{\V }(X)\) gives the free
  \(\Omega_{\V }\)-complete poset on a set \(X\). Put precisely, if
  \(X : \U \) is a set, then for every \(\Omega_{\V }\)-complete
  poset \((P,\order)\) with \(P : \U'\) and \(\order\) taking values
  in~\(\T'\) and function \(f : X \to P\), there exists a unique
  monotone \(\overline{f} : \lifting_{\V }(X) \to P\) preserving all
  suprema indexed by propositions in \(\V \), such that
  \[
    \begin{tikzcd}
      X \ar[dr, "\eta_X"'] \ar[rr, "f"] & & P \\
      & \lifting_{\V }(X) \ar[ur,dashed,"\overline{f}"']
    \end{tikzcd}
  \]
  commutes.
  \Omit{
    \(\overline{f} \circ \eta_X = f\).
  }
\end{theorem}
  \begin{proof}
    Similar to the proof of \cref{lifting-is-free}.
\end{proof}

Finally, a variation of \cref{construction-lifting} freely adds a least element
to a dcpo.

\begin{construction}[\(\lifting'_{\V }(D)\)]
  Let \(D : \DCPO{V}{U}{T}\) be a \(\V \)-dcpo. We construct a
  \emph{pointed} {\(\V\)-dcpo}
  \(\lifting'_{\V }(D) : \DCPO{V}{V^+\sqcup U}{V\sqcup T}\). Its
  carrier is given by the type
  \( \sum_{P : \V }\isprop(P) \times (P \to D) \).

  The order is given by
  \( (P,\varphi) \sqsubseteq_{\lifting'_{\V }(D)} (Q,\psi) \colonequiv \sum_{f :
    P \to Q}\prod_{p : P}\varphi(p) \below_{D} \psi(f(p)) \) and has a least
  element \((\emptyt,\emptyisprop,\uniquefromempty)\).

  Now let
  \(\alpha \colonequiv \paren*{Q_{(-)},\varphi_{(-)}} : I \to
  \lifting'_{\V }(D)\) be a directed family. Consider
  \(\Phi : \paren*{\sum_{i : I} Q_i}\to D\) given by
  \((i,q) \mapsto \varphi_i(q)\). The supremum of \(\alpha\) is given by
  \(\paren*{\exists_{i : I} Q_i , \psi}\), where \(\psi\) takes a witness that
  \(\sum_{i : I}Q_i\) is inhabited to the directed (for which we needed
  \(\exists_{i : I} Q_i\)) supremum \(\bigsqcup \Phi\) in \(D\).

  Finally, we write
  \(\eta'_{D} : D \to \lifting'_{\V }(D)\) for
  the continuous map
  \(x \mapsto \paren*{\unitt, \unitisprop, \lambda u.x} \). \lipicsEnd
\end{construction}
\begin{theorem}
  \label{lifting-is-free3}
  The construction \(\lifting'_{\V }(D)\) gives the free pointed
  \(\V \)-dcpo on a \(\V \)-dcpo \(D\). Put precisely, if
  \(D : \DCPO{V}{U}{T}\) is a \(\V \)-dcpo, then for every
  pointed \(\V \)-dcpo \(E : \DCPO{V}{U'}{T'}\) and continuous
  function \(f : D \to E\), there is a unique strict
  continuous function
  \(\overline{f} : \lifting'_{\V }(D) \to E\) such that
  \[
    \begin{tikzcd}
      D \ar[dr, "\eta'_{D}"'] \ar[rr, "f"] & & E  \\
      & \lifting'_{\V }(D) \ar[ur, dashed, "\overline{f}"']
    \end{tikzcd}
  \]
  commutes.
  \Omit{
    \(\overline{f} \circ \eta'_D = f\).
  }
\end{theorem}
  \begin{proof}
  Similar to the proof of \cref{lifting-is-free}.
\end{proof}
\subsection{Exponentials}
\begin{construction}[\(E^{D}\)]
  Let \(D\) and \(E\) be two \(\V \)-dcpos. We
  construct another \(\V \)-dcpo \(E^{D}\) as follows.
  Its carrier is given by the type of continuous functions from \(D\) to \(E\).

  These functions are ordered pointwise, i.e.\ if
  \(f,g : D \to E\), then
  \[
    f \order_{E^{D}} g \colonequiv \prod_{x : D}f(x)
    \order_{E} g(x).
  \]
  Accordingly, directed suprema are also given pointwise. Explicitly, let
  \(\alpha : I \to E^{D}\) be a directed family. For every
  \(x : D\), we have the family \(\alpha_x : I \to E\) given
  by \(i \mapsto \alpha_i(x)\). This is a directed family in \(E\) and
  so we have a well-defined supremum \(\bigsqcup\alpha_x : E\) for
  every \(x : D\). The supremum of \(\alpha\) is then given by
  \(x \mapsto \bigsqcup \alpha_x\), where one should check that this assignment
  is indeed continuous.

  Finally, if \(E\) is pointed, then so is \(E^{D}\),
  because, in that case, the function \(x \mapsto \bot_{E}\) is the
  least continuous function from \(D\) to \(E\). \lipicsEnd
\end{construction}
\begin{remark}
  In general, the universe levels of \(E^{D}\) can be quite large and
  complicated. For~if \(D : \DCPO{V}{U}{T}\) and \(D : \DCPO{V}{U'}{T'}\), then
  \( E^{D} : \DCPO{V}{V^+\sqcup U\sqcup T\sqcup U'\sqcup T'}{U\sqcup T'}
  \).  Even~if \(\V = \U \equiv \T \equiv \U' \equiv \T'\), the carrier of
  \(E^{D}\) still lives in the ``large'' universe \(\V ^+\). (Actually, this
  scenario cannot happen non-trivially in a predicative setting, since
  non-trivial dcpos cannot be ``small'' \cite{deJong2020}.) Even so, as observed
  in \cite{deJong2019}, if we take
  \(\U \equiv \T \equiv \U' \equiv \T' \equiv \U_1\) and \(\V \equiv \U_0\),
  then \(D , E , E^{D}\) are all elements of \(\DCPOnum{0}{1}{1}\).
\end{remark}

\section{Scott's \texorpdfstring{\(D_\infty\)}
  {D\_\textbackslash infty}} \label{D:infty}

We now construct, predicatively, Scott's famous pointed
dcpo \(D_\infty\) which is isomorphic to its own function space
\(D_\infty^{D_\infty}\)
(\cref{isomorphic-to-self-exponential}). We follow Scott's original paper
\cite{Scott1972} rather closely, but with two differences. Firstly, we
explicitly keep track of the universe levels to make sure that our constructions
go through predicatively. Secondly, \cite{Scott1972} describes sequential
(co)limits, while we study the more general directed (co)limits
(\cref{limits-and-colimits}) and then specialize to sequential (co)limits later
(\cref{Scott's-example-using-self-exponentiation}).

\subsection{Limits and Colimits}\label{limits-and-colimits}
\begin{definition}[Deflation]
  Let \(D\) be a dcpo. An endofunction \(f : D \to D\)
  is a \emph{deflation} if \(f(x) \sqsubseteq x\) for all \(x : D\).
\end{definition}
\begin{definition}[Embedding-projection pair]
  Let \(D\) and \(E\) be two dcpos. An
  \emph{embedding-projection pair} from \(D\) to \(E\)
  consists of two continuous functions
  \(\varepsilon : D \to E\)
  (the~\emph{embedding}) and
  \(\pi : E \to D\) (the~\emph{projection})
  such~that:
  \begin{enumerate}[(i)]
  \item \(\varepsilon\) is a section of \(\pi\);
  \item \(\varepsilon \circ \pi\) is a deflation.
  \end{enumerate}
\end{definition}

For the remainder of this section, fix the following setup. Let \(\V\), \(\U \),
\(\T\) and \(\W\) be type universes. Let \((I,\order)\) be a directed preorder
with \(I : \V\) and \(\order\) taken values in \(\W\). Suppose that we have:
\begin{enumerate}[(i)]
\item for every \(i : I\), a \(\V \)-dcpo
  \(D_i : \DCPO{V}{U}{T}\);
\item for every \(i,j : I\) with \(i \sqsubseteq j\), an embedding-projection
  pair \(\paren*{\varepsilon_{i,j},\pi_{i,j}}\) from \(D_i\) to
  \(D_j\);
\end{enumerate}
such that
\begin{enumerate}[(i)]
\item for every \(i : I\), we have \(\varepsilon_{i,i} = \pi_{i,i} = \id\);
\item for every \(i,j,k : I\) with \(i \sqsubseteq j \sqsubseteq k\), we have
  \(\varepsilon_{i,k} \sim \varepsilon_{j,k} \circ \varepsilon_{i,j}\) and
  \(\pi_{i,k} \sim \pi_{i,j} \circ \pi_{j,k}\).
\end{enumerate}

\begin{construction}[\(D_\infty\)]\label{D-infty}
  Given the above inputs, we construct another \(\V \)-dcpo \break
  {\({D_\infty} : {\DCPO{V}{U \sqcup V \sqcup W}{U \sqcup T}}\)} as follows.
  Its carrier is given by the type:
  \[
    \sum_{\sigma : \prod_{i : I} D_i} \prod_{j : I , i \sqsubseteq j}
    \pi_{i,j}\paren*{\sigma_j} = \sigma_i.
  \]
  These functions are ordered pointwise, i.e.\ if
  \(\sigma,\tau : I \to D_i\), then
  \[
    \sigma \sqsubseteq_{D_\infty} \tau \colonequiv \prod_{i : I}\sigma_i
    \sqsubseteq_{D_i}\tau_i.
  \]
  Accordingly, directed suprema are also given pointwise. Explicitly, let
  \(\alpha : A \to D_\infty\) be a directed family. For every
  \(i : I\), we have the family \(A \to D_i\) given by
  \(a \mapsto \paren*{\alpha(a)}_i\), and denoted by \(\alpha_i\). One can show
  that \(\alpha_i\) is directed and so we have a well-defined supremum
  \(\bigsqcup {\alpha_i} : D_i\) for every \(i : I\). The supremum of
  \(\alpha\) is then given by the function
  \(i : I \mapsto \bigsqcup {\alpha_i}\), where one should check that
  \(\pi_{i,j}\paren*{\bigsqcup {\alpha_j}} = \bigsqcup {\alpha_i}\) holds
  whenever \(i \sqsubseteq j\). \lipicsEnd
\end{construction}

\begin{remark}
  We allow for general universe levels here, which is why \(D_\infty\) lives
  in the relatively complicated universe \(\U \sqcup \V \sqcup \W\). In concrete
  examples, such as in \cref{Scott's-example-using-self-exponentiation}, the
  situation simplifies to \(\V = \W = \U_0\) and \(\U = \T = \U_1\).
\end{remark}

\begin{construction}[\(\pi_{i,\infty}\)]\label{pi-infty}
  For every \(i : I\), we have a continuous function
  \(\pi_{i,\infty} : {D_\infty \to D_i}\), given by
  \(\sigma \mapsto \sigma_i\). \lipicsEnd
\end{construction}
\begin{construction}[\(\varepsilon_{i,\infty}\)]\label{epsilon-infty}
  For every \(i,j : I\), consider the function
  \begin{align*}
    \kappa : D_i &\to {\paren*{\sum_{k : I} i
    \sqsubseteq k \times j \sqsubseteq k} \to D_j} \\
    \kappa_x(k,l_i,l_j) &= \pi_{i,j}\paren*{\varepsilon_{i,k}(x)}.
  \end{align*}
  Using directedness of \((I,\sqsubseteq)\), we can show that for every
  \(x : D_i\) the map \(\kappa_x\) is weakly constant (i.e.\
  all its values are equal). Therefore, we can apply
  \cite[Theorem~5.4]{KrausEtAl2017} and factor \(\kappa_x\) through
  \(\exists_{k : I} \paren*{i \sqsubseteq k \times j \sqsubseteq k}\). But
  \((I,\sqsubseteq)\) is directed, so
  \(\exists_{k : I} \paren*{i \sqsubseteq k \times j \sqsubseteq k}\) is a
  singleton. Thus, we obtain a function
  \(\rho_{i,j} : D_i \to D_j\) such that: if we are given
  \(k : I\) with \(\paren*{l_i,l_j} : i \sqsubseteq k \times j \sqsubseteq k\),
  then \(\rho_{i,j}(x) = \kappa_x(k,l_i,l_j)\).

  Finally, this allows us to construct for every \(i : I\), a continuous
  function \(\varepsilon_{i,\infty} : {D_i \to D_\infty}\) by
  mapping \(x : D_i\) to the function
  \(\lambda (j : I) . \rho_{i,j}(x)\). \lipicsEnd
\end{construction}
\begin{theorem}
  For every \(i : I\), the pair
  \(\paren*{\varepsilon_{i,\infty},\pi_{i,\infty}}\) is an embedding-projection
  pair.
\end{theorem}
\begin{lemma}
  Let \(i,j : I\) such that \(i \sqsubseteq j\). Then
  \(\pi_{i,j} \circ \pi_{j,\infty} \sim \pi_i\), and
  \(\varepsilon_{j,\infty} \circ \varepsilon_{i,j} \sim
  \varepsilon_{i,\infty}\).
\end{lemma}

\begin{theorem}\label{limit}
  The dcpo \(D_\infty\) with the maps
  \(\paren*{\pi_{i,\infty}}_{i : I}\) is the limit of 
  \(\paren*{\paren*{D_i}_{i : I} , \paren*{\pi_{i,j}}_{i \sqsubseteq
      j}}\). 
  That~is, given
  \begin{enumerate}[(i)]
  \item a \(\V \)-dcpo
    \(E : \DCPO{V}{U'}{T'}\),
  \item continuous functions \(f_i : E \to D_i\) for every
    \(i : I\),
  \end{enumerate}
  such~that \(\pi_{i,j} \circ f_j \sim f_i\) whenever \(i \sqsubseteq j\), we
  have a continuous function \(f_\infty : E \to D_\infty\)
  such~that \(\pi_{i,\infty} \circ f_\infty \sim f_i\) for every \(i :
  I\). Moreover, \(f_\infty\) is the unique such continuous function.

  The function \(f_\infty\) is given by mapping \(y : E\) to the
  function \(\lambda (i : I) . f_i(y)\).
\end{theorem}

\begin{theorem}\label{colimit}
  The dcpo \(D_\infty\) with the maps
  \(\paren*{\varepsilon_{i,\infty}}_{i : I}\) is the colimit of 
  \(\paren*{\paren*{D_i}_{i : I} , \paren*{\varepsilon_{i,j}}_{i
      \sqsubseteq j}}\). 
  That is, given
  \begin{enumerate}[(i)]
  \item a \(\V \)-dcpo
    \(E : \DCPO{V}{U'}{T'}\),
  \item continuous functions \(g_i : D_i \to E\) for every
    \(i : I\),
  \end{enumerate}
  such that \(g_j \circ \varepsilon_{i,j} \sim g_i\) whenever
  \(i \sqsubseteq j\), we have a continuous function
  \(g_\infty : D_\infty \to E\) such~that
  \(g_\infty \circ \varepsilon_{i,\infty} \sim g_i\) for every \(i :
  I\). Moreover, \(g_\infty\) is the unique such continuous function.

  The function \(g_\infty\) is given by
  \(\sigma \mapsto \bigsqcup_{i : I}g_i\paren*{\sigma_i}\), where one should
  check that the family \(i \mapsto g_i(\sigma_i)\) is indeed directed.
\end{theorem}
\begin{proof}
  For uniqueness, it is useful to know that an element \(\sigma : D_\infty\) can
  be expressed as the directed supremum
  \(\bigsqcup_{i : I} \varepsilon_{i,\infty}\paren*{\sigma_i}\). The rest can be
  checked directly.
\end{proof}
It should be noted that in both universal properties \(E\) can have its
  carrier in any universe \(\U'\) and its order taking values in any universe
  \(\T'\), even though we required all \(D_i\) to have their carriers and orders
  in two fixed universes \(\U \) and \(\T \), respectively.

\subsection{Scott's Example Using Self-exponentiation}
\label{Scott's-example-using-self-exponentiation}
We now show that we can construct Scott's
\(D_\infty\) \cite{Scott1972} predicatively. Formulated precisely, we construct a
pointed
\(D_\infty : \DCPOnum{0}{1}{1}\)
such that \(D_\infty\) is isomorphic to its
self-exponential~\(D_\infty^{D_\infty}\).

We employ the machinery from \cref{limits-and-colimits}. Following~\cite[pp.\
126--127]{Scott1972}, we inductively define pointed dcpos
\(D_n :\DCPOnum{0}{1}{1}\) for every natural number \(n\):
\begin{enumerate}[(i)]
\item \(D_0 \colonequiv \lifting_{\U_0}\paren*{\unitt_{\U_0}}; \)
\item \(D_{n+1} \colonequiv D_n^{D_n}\).
\end{enumerate}
Next, we inductively define embedding-projection pairs
\(\paren*{\varepsilon_n , \pi_n}\) from \(D_n\) to
\(D_{n+1}\):
\begin{enumerate}[(i)]
\item \(\varepsilon_0 : D_0 \to D_1\) is given by mapping
  \(x : D_0\) to the continuous function that is constantly~\(x\);
  \(\pi_0 : D_1 \to D_0\) is given by evaluating a continuous
  function \(f : D_0 \to D_0\) at \(\bot\);
\item \(\varepsilon_{n+1} : D_{n+1} \to D_{n+2}\) takes a
  continuous function \(f : D_n \to D_n\) to the continuous
  composite
  \(
    D_{n+1} \xrightarrow{\pi_n} D_n \xrightarrow{f} D_n
    \xrightarrow{\varepsilon_n} D_{n+1};
    \)

  \(\pi_{n+1} : D_{n+2} \to D_{n+1}\) takes a continuous
  function \(f : D_{n+1} \to D_{n+1}\) to the continuous
  composite
  \(
    D_n \xrightarrow{\varepsilon_n} D_{n+1} \xrightarrow{f}
    D_{n+1} \xrightarrow{\pi_n} D_n.
  \)
\end{enumerate}

In order to apply the machinery from \cref{limits-and-colimits}, we will need
embedding-projection pairs \(\paren*{\varepsilon_{n,m},\pi_{n,m}}\) from
\(D_n\) to \(D_m\) whenever \(n \leq m\). Let \(n\) and \(m\)
be natural numbers with \(n \leq m\) and let \(k\) be the natural number
\(m - n\). We define the pairs by induction on \(k\):
\begin{enumerate}[(i)]
\item if \(k = 0\), then we set \(\varepsilon_{n,n} = \pi_{n,n} = \id\);
\item if \(k = l + 1\), then
  \(\varepsilon_{n,m} = \varepsilon_{n+l} \circ \varepsilon_{n,n+l}\) and
  \(\pi_{n,m} = \pi_{n,n+l} \circ \pi_{n+l}\).
\end{enumerate}

So, \Cref{D-infty,pi-infty,epsilon-infty} give us
\(D_\infty : \DCPOnum{0}{1}{1}\) with embedding-projection pairs
\(\paren*{\varepsilon_{n,\infty},\pi_{n,\infty}}\) from \(D_n\) to
\(D_\infty\) for every natural number \(n\).

\begin{lemma}\label{pi-is-strict}
  Let \(n\) be a natural number. The function
  \(\pi_n : D_{n+1} \to D_n\) is strict. Hence, so is
  \(\pi_{n,m}\) whenever \(n \leq m\).
\end{lemma}
  \begin{proof}
  The first statement is proved by induction on \(n\). The second by induction
  on \(k\) with \(k \colonequiv m - n\).
\end{proof}

\begin{theorem}\label{isomorphic-to-self-exponential}
  The dcpo \(D_\infty\) is pointed and isomorphic to \(D_{\infty}^{D_{\infty}}\).
\end{theorem}
  \begin{proof}
  Since every \(D_n\) is pointed, we can consider the function
  \(\sigma : \prod_{n : \natt} D_n\) given by
  \(\sigma(n) \colonequiv \bot_{D_n}\). Then \(\sigma\) is an element of
  \(D_{\infty}\) by \cref{pi-is-strict} and it is the least, so \(D_\infty\) is
  indeed pointed.

  We start by constructing a continuous function
  \(\varepsilon'_\infty : {D_\infty \to D_\infty^{D_\infty}}\). By
  \cref{colimit}, it suffices to define continuous functions
  \(\varepsilon'_n : D_n \to D_\infty^{D_\infty}\) for every natural number
  \(n\) such that \(\varepsilon'_m \circ \varepsilon_{n,m} \sim \varepsilon'_n\)
  whenever \(n \leq m\). We do so as follows:
  \begin{enumerate}[(i)]
  \item
    \(\varepsilon'_{n+1} : D_{n+1} \colonequiv D_n^{D_n} \to
    D_\infty^{D_\infty}\) is given by mapping a continuous function
    \(f : D_n \to D_n\) to the continuous composite
    \( D_\infty \xrightarrow{\pi_{n,\infty}} D_n \xrightarrow{f} D_n
    \xrightarrow{\varepsilon_{n,\infty}} D_\infty ; \)
  \item
    \(\varepsilon'_0 : D_0 \to D_\infty^{D_\infty}\)
    is defined as the continuous composite
    \(D_0 \xrightarrow{\varepsilon_{0}} D_1 \xrightarrow
    {\varepsilon'_{1}} D_\infty\).
  \end{enumerate}

  Next, we construct a continuous function
  \(\pi'_\infty : {D_\infty^{D_\infty}} \to D_\infty\). By \cref{limit}, it
  suffices to define continuous functions
  \(\pi'_n : D_n \to D_\infty^{D_\infty}\) for every natural number \(n\) such
  that \(\pi_{n,m} \circ \pi'_m \sim \pi'_n\) whenever \(n \leq m\). We do so as
  follows:
  \begin{enumerate}[(i)]
  \item \(\pi'_{n+1} : D_\infty^{D_\infty} \to D_{n+1} \colonequiv D_n^{D_n}\)
    is given by mapping a continuous function \(f : D_\infty \to D_\infty\) to
    the continuous composite
    \( D_n \xrightarrow{\varepsilon_{n,\infty}} D_\infty \xrightarrow{f}
    D_\infty \xrightarrow{\pi_{n,\infty}} D_n\);
  \item \(\pi'_0 : D_\infty^{D_\infty} \to D_0\) is defined as the continuous
    composite
    \(D_\infty \xrightarrow{\pi'_{1}} D_1 \xrightarrow {\pi_{0}} D_0\).
  \end{enumerate}

  It remains to prove that \(\varepsilon'_\infty\) and \(\pi'_\infty\) are
  inverses. To this end, it is convenient to have an alternative description of
  the maps \(\varepsilon'_\infty\) and \(\pi'_\infty\).
  \begin{equation}
    \text{For every \(\sigma : D_\infty\), we have
    \(\varepsilon'_\infty(\sigma) = \bigsqcup_{n : \natt}\varepsilon'_{n+1}
    \paren*{\sigma_{n+1}}\).}
  \end{equation}
  \begin{equation}
    \text{For every continuous \(f : D_\infty \to D_\infty\), we
    have
    \(\pi'_\infty(f) = \bigsqcup_{n : \natt} \epsilon_{n+1,\infty}
      \paren*{\pi'_{n+1} \paren*{f}}\).}
  \end{equation}
  Using these equations we can prove that \(\varepsilon'_\infty\) and
  \(\pi'_\infty\) are inverses exactly as in~\cite[Proof~of~Theorem~4.4]{Scott1972}.
\end{proof}

\begin{remark}
  Of course, \cref{isomorphic-to-self-exponential} is only interesting in case
  \(D_\infty \not\simeq \unitt\). Fortunately,
  \(D_\infty\) has (infinitely) many elements besides
  \(\bot_{D_\infty}\). For instance, we can consider
  \(x_0 \colonequiv \eta(\star) : D_0\) and
  \(\sigma_0 : \prod_{n : \natt} D_n\) given by
  \(\sigma_0(n) \colonequiv \varepsilon_{0,n}(x_0)\). Then, \(\sigma_0\) is an
  element of \(D_\infty\) not equal to~\(\bot_{D_\infty}\),
  because \(x_0 \neq \bot_{D_0}\).
\end{remark}
\section{Continuous and Algebraic Dcpos} \label{continuous:and:algebraic:dcpos}
We next consider dcpos generated by certain elements called compact, or more
generally generated by a certain way-below relation, giving rise to algebraic
and continuous domains.

\subsection{The Way-below Relation}
\begin{definition}[Way-below relation, \(x \ll y\)] Let \(D\) be a
  \(\V \)-dcpo and \(x,y : D\). We say that
  \(x\)~is~\emph{way below}~\(y\), denoted by \(x \ll y\), if for every
  \(I : \V \) and directed family \(\alpha : I \to D\), whenever
  we have \(y \order \bigsqcup \alpha\), then there exists some element
  \(i : I\) such that \(x \order \alpha_i\) already. Symbolically,
  \[
    x \ll y \colonequiv \prod_{I : \V }\prod_{\alpha : I \to D}
    \paren*{ \isdirected(\alpha) \to y \order {\bigsqcup} \alpha \to \exists_{i
        : I}\, x \order \alpha_i }.
  \]
\end{definition}
\begin{lemma}\label{way-below-properties}
  The way-below relation enjoys the following properties.
  \begin{enumerate}[(i)]
  \item It is proposition-valued.
  \item If \(x \ll y\), then \(x \order y\).
  \item If \(x \order y \ll v \order w\), then \(x \ll w\).
  \item It is antisymmetric.
  \item It is transitive.
  \end{enumerate}
\end{lemma}

\begin{lemma}\label{order-in-terms-of-way-below}
  Let \(D\) be a dcpo. Then \(x \order y\) implies
  \(\prod_{z : D} \paren*{z \ll x \to z \ll y}\).
\end{lemma}
\begin{proof}
  By \cref{way-below-properties}(iii).
\end{proof}
\begin{definition}[Compact]
  Let \(D\) be a dcpo. An element \(x : D\) is called
  \emph{compact} if \(x \ll x\).
\end{definition}
\begin{example}
  The least element of a pointed dcpo is always compact.
\end{example}
\begin{example}[Compact elements in \(\lifting_{\V }(X)\)]
  \label{compact-elements-in-lifting}
  Let \(X : \U \) be a set. An element
  \((P,\varphi) : \lifting_{\V }(X)\) is compact if and only if \(P\) is
  decidable.
    \begin{claimproof}
    Suppose that \((P,\varphi)\) is compact. We must show that \(P\) is
    decidable. Consider the family
    \(\alpha : \paren*{\unitt_{\U } + P} \to \lifting_{\V }(X)\)
    given by \(\inl(\star) \mapsto \bot\) and \(\inr(p) \mapsto
    (P,\varphi)\). This is directed, so \(\alpha\) has a supremum in
    \(\lifting_{\V }(X)\). Observe that
    \((P,\varphi) \order \bigsqcup \alpha\) holds. Hence, by assumption that
    \((P,\varphi)\) is compact, we have
    \(\exists_{i : {\unitt_{\U } + P}} \paren*{(P,\varphi) \order
      \alpha_i}\). Since decidability of \(P\) is a proposition, we obtain
    \(i : {\unitt_{\U } + P}\) such that \((P,\varphi) \order
    \alpha_i\). There are two cases: \(i = \inl(\star)\) or \(i = \inr(p)\). In
    the first case, \((P,\varphi) \order \bot\), so \(\lnot P\). In the second
    case, we have \(P\). Hence, \(P\) is decidable.

    Conversely, suppose that we have \((P,\varphi) : \lifting_{\V }(X)\)
    with \(P\) decidable. There are two cases: either \(\lnot P\) or \(P\). If
    \(\lnot P\), then \((P,\varphi)\) is the least element of
    \(\lifting_{\V }(X)\), so it is compact. If \(P\), then let
    \(\alpha : I \to \lifting_{\V }(X)\) be a directed family with
    \((P,\varphi) \order \bigsqcup \alpha\).  Since \(P\) holds, we get the
    equality \((P,\varphi) = \bigsqcup \alpha\) and
    \(\paren*{\bigsqcup\alpha}\isdefined\). Recalling
    \cref{construction-lifting}, this means that we have
    \(\exists_{i : I} \paren*{\alpha_i} \isdefined\). Hence,
    \(\exists_{i : I} \paren*{(P,\varphi) \order \alpha_i}\) holds as well,
    finishing the proof.
  \end{claimproof}
\end{example}
\begin{definition}[Kuratowski finite]
  A type \(X\) is \emph{Kuratowski finite} if there exists some natural number
  \(n : \natt\) and a surjection \(e : \Fin(n) \twoheadrightarrow X\).
\end{definition}
That is, \(X\) is \emph{Kuratowski finite} if its elements can be finitely
enumerated, possibly with repetitions.

\begin{example}[Compact elements in \(\powerset_{\U }(X)\)]
  \label{compact-elements-in-powerset}
  Let \(X : \U \) be a set. An element \(A : \powerset_{\U }(X)\)
  is compact if and only if its total type \(\mathbb T A\) is Kuratowski finite.
  \begin{claimproof}
    Write \(\iota : \List(X) \to \powerset_{\U}(X)\) for the map that regards a
    list on \(X\) as a subset of \(X\). The inductively generated type
    \(\List(X)\) of lists on \(X\) lives in the same universe \(\U\) as \(X\).

    Suppose that \(A\) is compact. We must show that \(\mathbb T(A)\) is
    Kuratowski finite. Consider the map
    \(\alpha : \List(\mathbb T(A)) \to \powerset_{\U (X)}\) which takes a list
    \([\paren*{x_0,p_0},\dots,\paren*{x_{n-1},p_{n-1}}]\) to
    \(\iota\paren*{[x_0,\dots,x_{n-1}]}\). Since the empty list is an element of
    \(\List(\mathbb T(A))\) and because we can concatenate lists, \(\alpha\) is
    directed. Moreover, \(\List(\mathbb T(A)) : \U \) and
    \(A = \bigsqcup \alpha\) holds. Hence, by compactness, there exists some
    \(l \colonequiv [\paren*{x_0,p_0},\dots,\paren*{x_{n-1},p_{n-1}}] :
    \List(\mathbb T(A))\) such that \(A \subseteq \alpha(l)\) already. Hence,
    the map \(m : \Fin(n) \mapsto (x_m,p_m) : \mathbb T(A)\) is a surjection, so
    \(\mathbb T(A)\) is Kuratowski finite.

    Conversely, suppose that \(A\) is a subset such that \(\mathbb T(A)\) is
    Kuratowski finite. We must prove that it is compact. Let
    \(B_{(-)} : I \to \powerset_{\U }(X)\) be directed such that
    \(A \subseteq \bigsqcup_{i : I} B_i\). Since
    \(\exists_{i : I} A \subseteq B_i\) is a proposition, we can use Kuratowski
    finiteness of \(\mathbb T(A)\) to obtain a natural number \(n\) and a
    surjection \(e : \Fin(n) \twoheadrightarrow \mathbb T(A)\). For each
    \(m : \Fin(n)\), find \(i_m\) such that \(e_m \in B_{i_m}\). By directedness
    of \(I\), there exists \(k : I\) such that \(e_m \in B_k\) for every
    \(m : \Fin(n)\). Hence, \(\exists_{k : I}\,A \subseteq B_k\), as desired.
  \end{claimproof}
\end{example}

\subsection{Continuous Dcpos}
Classically, a continuous dcpo is a dcpo where every element is the directed
join of the set of elements way below it \cite{AmadioCurien1998}. Predicatively,
we must be careful, because if \(x\) is an element of a dcpo \(D\), then
\(\sum_{y : D} y \ll x\) is typically large, so its directed join need not exist
for size reasons. Our solution is to use a predicative version of bases
\cite{AbramskyJung1994} that accounts for size issues. For the special case of
algebraic dcpos, our situation is the poset analogue of accessible categories
\cite{AdamekRosicky1994}. Indeed, in category theory requiring smallness is
common, even in impredicative settings, see for instance
\cite{JohnstoneJoyal1982}, where continuous dcpos are generalized to continuous
categories.

\begin{definition}[Basis, approximating family]
  A \emph{basis} for \(\V \)-dcpo \(D\) is a function
  \(\beta : B \to D\) with \(B : \V \) such that for every
  \(x : D\) there exists some \(\alpha : I \to B\) with
  \(I : \V \) such that
  \begin{enumerate}[(i)]
  \item \(\beta \circ \alpha\) is directed and its supremum is \(x\);
  \item \(\beta(\alpha_i) \ll x\) for every \(i : I\).
  \end{enumerate}
  We summarise these requirements by saying that \(\alpha\) is an
  \emph{approximating family} for \(x\).

  Moreover, we require that \(\ll\) is small when restricted to the basis. That
  is, we have \({\ll^B} : {B \to B \to \V }\) such that
  \(\paren*{\beta(b) \ll \beta(b')} \simeq \paren*{b \ll^B b'}\) for every
  \(b,b' : B\).
\end{definition}

\begin{definition}[Continuous dcpo]
  A dcpo \(D\) is \emph{continuous} if there exists some basis for it.
\end{definition}

We postpone giving examples of continuous dcpos until we have developed the
theory further, but the interested reader may look ahead to
\Cref{lifting-is-algebraic,powerset-is-algebraic,ideal-completion-of-dyadics}.

A useful property of bases is that it allows us to express the order fully in
terms of the way-below relation, giving a converse to
\cref{order-in-terms-of-way-below}.
\begin{lemma}\label{order-in-terms-of-way-below'}
  Let \(D\) be a dcpo with basis \(\beta : B \to D\). Then
  \(x \order y\) holds if and only if
  \(\prod_{b : B} \paren*{\beta(b) \ll x \to \beta(b) \ll y}\).
\end{lemma}
  \begin{proof}
  The left-to-right implication holds by \cref{order-in-terms-of-way-below}. For
  the converse, suppose that we have \(x,y : D\) such that for every
  \(\prod_{b : B}\paren*{\beta(b) \ll x \to \beta(b) \ll y}\). Since
  \(x\order y\) is a proposition, we can obtain \(\alpha : I \to B\) such that
  \(\beta \circ \alpha\) is directed and \(\bigsqcup \beta \circ \alpha = x\)
  and \(\beta(\alpha_i) \ll x\) for every \(i : I\). It then suffices to show
  that \(\bigsqcup \beta \circ \alpha \order y\). Since \(\bigsqcup\) gives the
  \emph{least} upper bound, it is enough to prove that
  \(\beta(\alpha_i) \order y\) for every \(i : I\), but this holds by our
  hypothesis, our assumption that \(\beta(\alpha_i) \ll x\) for every
  \(i : I\), and \cref{way-below-properties}(ii).
\end{proof}

\begin{lemma}\label{order-small-on-basis}
  Let \(D\) be a \(\V \)-dcpo with a basis
  \(\beta : B \to D\). Then \(\order\) is small when restricted to the
  basis, i.e.\ \(\beta(b_1) \order \beta(b_2)\)~has~size~\(\V \) for
  every two elements \(b_1,b_2 : B\). Hence, we have
  \({\order^B} : {B \to B \to \V }\) such that
  \(\prod_{b_1,b_2 : B} \paren*{b_1\order^B b_2} \simeq \paren*{\beta(b_1)
    \order \beta(b_2)}\).
\end{lemma}
  \begin{proof}
  Let \(b_1,b_2 : B\) and note that we have the following equivalences:
  \begin{align*}
    \paren*{\beta(b_1) \order \beta(b_2)} & \simeq
    \prod_{b : B}\paren*{\beta(b) \ll \beta(b_1) \to \beta(b) \ll \beta(b_2)}
    && \text{(by \cref{order-in-terms-of-way-below'})} \\
    & \simeq \prod_{b : B}\paren*{b \ll^B b_1 \to b \ll^B b_2}
    && \text{(by definition of a basis)},
  \end{align*}
  but the latter is a type in \(\V \).
\end{proof}
The most significant properties of a basis are the interpolation properties. We
consider nullary, unary and binary versions here. The binary interpolation
property actually follows fairy easily from the unary one, but we still record
it, because we wish to show that bases are examples of the abstract bases that
we define later (cf.\ \cref{bases-are-abstract-bases}). Our proof of unary
interpolation is a predicative version of \cite{Escardo2000}.

\begin{lemma}[Nullary interpolation]
  \label{nullary-interpolation}
  Let \(D\) be a dcpo with a basis \(\beta : B \to D\). For
  every \(x : D\), there exists some \(b : B\) such that
  \(\beta(b) \ll x\).
\end{lemma}
\begin{proof}
  Immediate from the definitions of a basis and a directed family.
\end{proof}
\begin{lemma}[Unary interpolation]
  \label{unary-interpolation}
  Let \(D\) be a \(\V \)-dcpo with basis
  \(\beta : B \to D\) and let \(x,y : {D}\). If \(x \ll y\),
  then there exists some \(b : B\) such that \(x \ll \beta(b) \ll y\).
\end{lemma}
\begin{proof}
  Let \(x,y : D\) with
  \(x \ll y\). Since \(\beta\) is a basis, there exists an approximating family
  \(\alpha : I \to B\) for \(y\). Consider the family
  \begin{equation*}\label{auxiliary-family}
    \paren*{K \colonequiv {\sum_{b : B} \sum_{i : I} b \ll^B \alpha_i} : \V }
    \xrightarrow{\fst} B \xrightarrow{\beta} D.
    \tag{\(\dagger\)}
  \end{equation*}
  \begin{claim*}
    The family \eqref{auxiliary-family} is directed.
  \end{claim*}
  \begin{claimproof}
    By directedness of \(\alpha\) and nullary interpolation, the type \(K\) is
    inhabited.

    Now suppose that we have \(b_1,b_2 : B\) and \(i_1,i_2 : I\) with
    \(b_1 \ll^B \alpha_{i_1}\) and \(b_2 \ll^B \alpha_{i_2}\). By directedness
    of \(\alpha\), there exists \(k : I\) with
    \(\alpha_{i_1},\alpha_{i_2} \order^B \alpha_{k}\). Since \(\beta\) is a
    basis for \(D\), there exists an approximating family \(\gamma : J \to B\)
    for \(\beta\paren*{\alpha_k}\). From \(b_1 \ll^B \alpha_{i_1}\) we obtain
    \(b_1 \ll^B \alpha_k\) and similarly, \(b_2 \ll^B \alpha_k\). Hence, there
    exist \(j_1,j_2 : J\) such that \(b_1 \order^B \gamma_{j_1}\) and
    \(b_2 \order^B \gamma_{j_2}\). By directedness of \(J\), there exists
    \(m : J\) with \(\gamma_{j_1},\gamma_{j_2} \order^B \gamma_{m}\). Thus,
    putting this all together, we see that:
    \( b_1,b_2 \order^B \gamma_{j_m} \ll^B \alpha_k.  \) Hence,
    \eqref{auxiliary-family} is directed.
  \end{claimproof}

  Thus, \eqref{auxiliary-family} has a supremum \(s\) in \(D\).

  \begin{claim*}
    We have \(y \order s\).
  \end{claim*}
  \begin{claimproof}
    Since \(y = \bigsqcup \beta \circ \alpha\), it suffices to prove that
    \(\beta\paren*{\alpha_i} \order s\) for every \(i : I\). Let \(i : I\) be
    arbitrary and let \(\gamma_j : J \to B\) be some approximating family for
    \(\beta\paren*{\alpha_i}\). Then it is enough to establish
    \(\beta\paren*{\gamma_j} \order s\) for every \(j : J\). But we know that
    \(\gamma_j \ll^B \alpha_i\), so \(\beta_{\gamma_j} \order s\) by definition
    of \eqref{auxiliary-family} and the fact that \(s\) is the supremum of
    \eqref{auxiliary-family}.
  \end{claimproof}

  Finally, from \(y \order s\) and \(x \ll y\), it follows that there must exist
  \(b : B\) and \(i : I\) such that:
  \(x \order \beta(b) \ll \beta(\alpha_i) \ll y\), which finishes the proof.
\end{proof}
\begin{lemma}[Binary interpolation]
  \label{binary-interpolation}
  Let \(D\) be a \(\V \)-dcpo with basis
  \(\beta : B \to D\) and let \(x,y,z : {D}\). If
  \({x,y} \ll z\), then there exists some \(b : B\) such that
  \({x,y} \ll \beta(b) \ll z\).
\end{lemma}
\begin{proof}
  Let \(x,y,z : D\) such that \(x,y \ll z\). By unary interpolation,
  there are \(b_x, b_y : B\) such that \(x \ll \beta(b_x) \ll z\) and
  \(y \ll \beta(b_y) \ll z\). Since \(\beta\) is a basis, there exists a family
  \(\alpha : I \to B\) such that \(\beta(\alpha_i) \ll z\) for every \(i : I\),
  and \(\beta \circ \alpha\) is directed and has supremum \(z\). Since
  \(\beta(b_x) \ll z\), there must exists \(i_x : I\) with
  \(\beta(b_x) \order \beta\paren*{\alpha_{i_x}}\). Similarly, there exists
  \(i_y : I\) such that \(\beta(b_y) \order \beta\paren*{\alpha_{i_y}}\). By
  directedness of \(\beta \circ \alpha\), there exists \(k : I\) with
  \(\beta\paren*{\alpha_{i_x}},\beta\paren*{\alpha_{i_y}} \order
  \beta\paren*{\alpha_k}\). Hence,
  \[
    x \ll {\beta\paren*{b_x}} \order {\beta\paren*{\alpha_{i_x}}} \order
    {\beta\paren*{\alpha_k}} \ll z \qquad\text{and}\qquad y \ll
    {\beta\paren*{b_y}} \order {\beta\paren*{\alpha_{i_y}}} \order
    {\beta\paren*{\alpha_k}} \ll z,
  \]
  so that \(x,y\ll \beta\paren*{\alpha_k} \ll z\), as wished.
\end{proof}

\subsection{Algebraic Dcpos}
We now turn to a particular class of continuous dcpos, called algebraic dcpos.
\begin{definition}[Algebraic dcpo]
  A dcpo \(D\) is \emph{algebraic} if there exists some basis
  \(\beta : B \to D\) for it such that \(\beta(b)\) is compact for
  every \(b : B\).
\end{definition}

\begin{lemma}\label{algebraic-criterion}
  Let \(D\) be a \(\V \)-dcpo. Then \(D\) is algebraic
  if and only if there exists \(\beta : B \to D\) with
  \(B : \V \) such that
  \begin{enumerate}[(i)]
  \item every element \(\beta(b)\) is compact;
  \item for every \(x : D\), there exists \(\alpha : I \to B\) with \(I : \V \)
    such that \(\beta \circ \alpha\) is directed and
    \(x = \bigsqcup \beta \circ \alpha\).
  \end{enumerate}
\end{lemma}
  \begin{proof}
  We just need to show that having \(\beta : B \to D\) and
  \(\alpha : I \to B\) such that every element \(\beta(b)\) is compact and
  \(x = \bigsqcup \beta \circ \alpha\), already implies that
  \(\beta\paren*{\alpha_i} \ll x\) for every \(i : I\). But if \(i : I\), then
  \(\beta\paren*{\alpha_i} \ll \beta\paren*{\alpha_i} \order \bigsqcup \beta
  \circ \alpha = x\) by compactness of \(\beta\paren*{\alpha_i}\), so
  \cref{way-below-properties}(iii) now finishes the proof.
\end{proof}
\begin{example}[\(\lifting_{\U }(X)\) is algebraic]
  \label{lifting-is-algebraic}
  Let \(X : \U \) be a set and consider
  \(\lifting_{\U }(X) : \DCPO{U}{U^+}{U^+}\). The basis
  \({[\bot , \eta_X ]} : {\paren*{\unitt_{\U } + X} \to
    \lifting_{\U }(X)}\) exhibits \(\lifting_{\U }(X)\) as an
  algebraic dpco.

  \begin{claimproof}
    By \cref{compact-elements-in-lifting}, the elements \(\bot\) and
    \(\eta_X(x)\) (with \(x : X\)) are all compact, so it remains to show that
    \(\unitt_{\U } + X\) is indeed a basis. Recalling
    \Cref{pointed-dcpos-have-subsingleton-joins,lifting-element-as-sup}, we can
    write any element \((P,\varphi) : \lifting_{\V}(X)\) as the directed join
    \(\bigsqcup \paren*{[\bot,\eta_X] \circ \alpha}\) with
    \(\alpha \colonequiv [\id,\varphi] : \paren*{\unitt_{\U } + P} \to
    \paren*{\unitt_{\U } + X}\). By \cref{algebraic-criterion} the proof is
    finished.
  \end{claimproof}
\end{example}

\begin{example}[\(\powerset_{\U }(X)\) is algebraic]
  \label{powerset-is-algebraic}
  Let \(X : \U \) be a set and consider
  \(\powerset_{\U }(X) : \DCPO{U}{U^+}{U}\). The basis
  \({\iota} : {\List(X) \to \powerset_{\U }(X)}\) that maps a finite list to a
  Kuratowski finite subset exhibits \(\powerset_{\U }(X)\) as an algebraic dpco.
  \begin{claimproof}
    By \cref{compact-elements-in-powerset}, the element \(\iota(l)\) is compact
    for every list \(l : \List(X)\), so it remains to show that \(\List(X)\) is
    indeed a basis.  In the proof of \cref{compact-elements-in-powerset}, we saw
    that every \(\U \)-subset~\(A\) of \(X\) can be expressed as the directed
    supremum \(\bigsqcup \iota \circ \alpha\) where
    \(\alpha : \List(\mathbb T(A)) \to \List(X)\) takes a list
    \([\paren*{x_0,p_0},\dots,\paren*{x_{n-1},p_{n-1}}]\) to the list
    \([x_0,\dots,x_{n-1}]\). Another application of \cref{algebraic-criterion}
    now finishes the proof.
  \end{claimproof}
\end{example}

\begin{example}[Scott's \(D_\infty\) is algebraic]
  \label{D-infty-is-algebraic}
  The pointed dcpo \(D_\infty : \DCPOnum{0}{1}{1}\) with
  \(D_\infty \cong D_\infty^{D_\infty}\) from
  \cref{Scott's-example-using-self-exponentiation} is algebraic.
    We postpone the proof until \cref{proof-of-D-infty-is-algebraic}, since we
    will need some additional results on locally small dcpos.
\end{example}

\subsection{Ideal Completion}
Finally, we consider how to build dcpos from posets, or more generally from
abstract bases, using the rounded ideal completion
\cite[Section~2.2.6]{AbramskyJung1994}.
Given our definition of the notion of dcpo, the reader might expect us to define
ideals using families rather than subsets. However, we use subsets for
extensionality reasons. Two subsets \(A\) and \(B\) of some \(X\) are equal
exactly when \({x\in A \iff x\in B}\) for every \(x : X\). However, given two
(directed) families \(\alpha : I \to X\) and \(\beta : J \to X\), it is of
course not the case (it does not even typecheck) that \(\alpha = \beta\) when
\({\Pi_{i : I}\exists_{j : J}\alpha_i = \beta_j}\) and
\(\Pi_{j : J}\exists_{i : I}\beta_j = \alpha_i\) hold. We could try to construct
the ideal completion by quotienting the families, but then it seems impossible
to define directed suprema in the ideal completion without resorting to choice.

\begin{definition}[Abstract basis]
  A pair \((B,\prec)\) with \(B : \V \) and \(\prec\) taking values in
  \(\V \) is called a \(\V \)-\emph{abstract basis} if:
  \begin{enumerate}[(i)]
  \item \(\prec\) is proposition-valued;
  \item \(\prec\) is transitive;
  \item \(\prec\) satisfies nullary interpolation, i.e.\ for every \(x : B\),
    there exists some \(y : B\) with \(y \prec x\);
  \item \(\prec\) satisfies binary interpolation, i.e.\ for every \(x,y : B\)
    with \(x \prec y\), there exists some \(z : B\) with \(x \prec z \prec y\).
  \end{enumerate}
\end{definition}
\begin{example}\label{bases-are-abstract-bases}
  Let \(D\) be a \(\V \)-dcpo with a basis \(\beta : B \to D\), By
  \Cref{way-below-properties,nullary-interpolation,binary-interpolation},
  the~pair \(\paren*{B,\ll^B}\) is an example of a \(\V \)-abstract basis.
\end{example}
\begin{example}\label{preorder-is-abstract-basis}
  Any preorder \((P,\order)\) with \(P : \V \) and \(\order\) taking
  values in \(\V \) is a \(\V \)-abstract basis, since reflexivity
  implies both interpolation properties.
\end{example}

For the remainder of this section, fix some arbitrary \(\V \)-abstract
basis \((B,\prec)\).

\begin{definition}[Directed subset]
  Let \(A\) be a \(\V \)-subset of \(B\). Then \(A\) is \emph{directed} if \(A\)
  is inhabited (i.e.\ \(\exists_{x : B}\, x \in A\) holds) and for every
  \(x,y \in A\), there exists some \(z \in A\) such that \(x,y\order z\).
\end{definition}

\begin{definition}[Ideal, lower set]
  Let \(A\) be a \(\V \)-subset of \(B\). Then \(A\) is an \emph{ideal} if \(A\)
  is a directed subset of \(B\) and \(A\) is a \emph{lower set}, i.e.\ if
  \(x \prec y\) and \(y \in A\), then \(x \in A\) as well.
\end{definition}

\begin{construction}[Rounded ideal completion \(\Idl(B,\prec)\)]
  We construct a \(\V\)-dcpo, known as the \emph{(rounded) ideal completion}
  \({\Idl(B,\prec)} : {\DCPO{V}{V^+}{V}}\) of \((B,\prec)\). The carrier is
  given by the type \( \sum_{I : B \to \Omega_{\V} }\isideal(I) \) of ideals on
  \((B,\prec)\).
  The order is given by subset inclusion~\(\subseteq\). If~we have a directed
  family \(\alpha : A \to \Idl(B,\prec)\) of ideals (with \(A : \V \)), then the
  subset given by \(\lambda x . \exists_{a : A} x \in \alpha_a\) is again an
  ideal and the supremum of \(\alpha\) in \(\Idl(B,\prec)\). \lipicsEnd
\end{construction}

\begin{lemma}[Rounded ideals]
  \label{ideals-are-rounded}
  The ideals of \(\Idl(B,\prec)\) are \emph{rounded}. That is, if
  \(I : \Idl(B,\prec)\) and \(x \in I\), then there exists some \(y \in I\) with
  \(x \prec y\).
\end{lemma}
  \begin{proof}
  Immediate from the fact that ideals are directed sets.
\end{proof}
\begin{definition}[Principal ideal \(\principalideal x\)]
  We write \(\principalideal(-) : B \to \Idl(B,\prec)\) for the map that takes
  \(x : B\) to the \emph{principal ideal} \(\lambda y . y \prec x\).
\end{definition}

\begin{lemma}\label{ideal-as-sup-of-principal-ideals}
  Let \(I : \Idl(B,\prec)\) be an ideal. Then \(I\) may be expressed as the
  supremum of the directed family
  \((x,p) : {\mathbb T(I)} \mapsto {\principalideal{x}} : {\Idl(B,\prec)}\),
  which we will denote by \(I = \bigsqcup_{x \in I} \principalideal{x}\).
\end{lemma}
  \begin{proof}
  Directedness of the family follows from the fact that \(I\) is a directed
  subset. Since \(I\) is a lower set, \(\principalideal{x} \subseteq I\) holds
  for every \(x \in I\), establishing
  \(\bigsqcup_{x \in I} \principalideal{x} \subseteq I\). The reverse inclusion
  follows from \cref{ideals-are-rounded}.
\end{proof}
We wish to prove that \(\Idl(B,\prec)\) is continuous with basis
\(\downarrow(-) : B \to \Idl(B,\prec)\). To this end, it is useful to
express \(\ll_{\Idl(B,\prec)}\) in more elementary terms.

\begin{lemma}\label{ideals-way-below-in-terms-of-containment}
  Let \(I , J : \Idl(B,\prec)\) be two ideals. Then \(I \ll J\) holds if and
  only there exists \(x \in J\) such that \(I \subseteq \principalideal{x}\).
\end{lemma}
  \begin{proof}
  The left-to-right implication follows immediately from
  \cref{ideal-as-sup-of-principal-ideals}.

  For the converse, note that \(I \ll J\) is a proposition, so we may assume
  that we have \(x \in J\) with \(I \subseteq \principalideal{x}\). Now let
  \(\alpha : A \to \Idl(B,\prec)\) be a directed family such that
  \(J \subseteq \bigsqcup \alpha\). Then there must exist some \(a : A\) for
  which \(x \in \alpha_a\). But \(I \subseteq \principalideal{x}\) and
  \(\alpha_a\) is a lower set, so \(I \subseteq \alpha_a\).
\end{proof}
\begin{theorem}\label{Idl-is-continuous}
  The map \({\downarrow(-)} : {B \to \Idl(B,\prec)}\) is a basis for
  \(\Idl(B,\prec)\). Thus, \(\Idl(B,\prec)\) is a continuous \(\V \)-dcpo.
\end{theorem}
  \begin{proof}
  Let \(I : \Idl(B,\prec)\) be arbitrary. By
  \cref{ideal-as-sup-of-principal-ideals} we can express \(I\) as the supremum
  \(\bigsqcup_{x \in I} \principalideal{x}\), so it is enough to prove that
  \(\principalideal{x} \ll I\) for every \(x \in I\). But this follows from
  \Cref{ideals-way-below-in-terms-of-containment,ideals-are-rounded}.
\end{proof}
  \begin{lemma}\label{reflexive-implies-algebraic}
  If \(\prec\) is reflexive, then the compact elements of \(\Idl(B,\prec)\) are
  exactly the principal ideals and \(\Idl(B,\prec)\) is algebraic.
\end{lemma}
\begin{proof}
  Immediate from \cref{ideals-way-below-in-terms-of-containment}.
\end{proof}
\begin{theorem}\label{free-dcpo-on-a-poset}
  The ideal completion is the free dcpo on a small poset. That is, if we
  have a poset \((P,\order)\) with \(P : \V \) and \(\order\) taking
  values in \(\V \), then for every \(D : \DCPO{V}{U}{T}\) and
  monotone function \(f : P \to D\), there is a unique continuous
  function \(\overline{f} : {\Idl(P,\order) \to D}\) such that
  \[
    \begin{tikzcd}
      P \ar[dr, "\principalideal(-)"'] \ar[rr, "f"] & & D \\
      & \Idl(P,\order) \ar[ur,"\overline{f}"',dashed]
    \end{tikzcd}
  \]
  commutes.
\end{theorem}
  \begin{proof}
  Given \((P,\order)\), \(D\) and \(f\) as in the theorem, we define
  \(\overline{f}\) by mapping an ideal \(I\) to the supremum of the directed
  (since \(I\) is an ideal) family
  \(\mathbb {T}(I) \xrightarrow{\fst} P \xrightarrow{f} D \).

  Commutativity of the diagram expresses that
  \(f(x) = \bigsqcup_{y \order x} f(y)\) for every \(x : P\). By~antisymmetry of
  \(\order\), it suffices to prove \(f(x) \order \bigsqcup_{y \order x}f(y)\)
  and \(\bigsqcup_{y \order x}f(y) \order f(x)\). The first holds by reflexivity
  of \(\order\) and the second holds because \(f\) is monotone.

  Uniqueness of \(\overline{f}\) follows easily using
  \cref{ideal-as-sup-of-principal-ideals}. Finally, continuity of
  \(\overline{f}\) is not hard to establish either.
\end{proof}

\begin{definition}[Continuous retract, section, retraction]
  A \(\V \)-dcpo \(D\) is a \emph{continuous retract} of another
  \(\V \)-dcpo \(E\) if we have continuous functions
  \(s : D \to E\) (the \emph{section}) and
  \(r : E \to D\) (the \emph{retraction}) such that
  \(r(s(x)) = x\) for every \(x : D\).
\end{definition}
\begin{theorem}
  If \(E\) is a dcpo with basis \(\beta : B \to D\) and
  \(D\) is a continuous retract of \(E\) with retraction
  \(r\), then \(r \circ \beta\) is a basis for \(D\).
\end{theorem}
  \begin{proof}
  Let \(E\) be a dcpo with basis \(\beta : B \to D\) and suppose that we have
  continuous retraction \(r : E \to D\) with continuous section \(s : D \to
  E\). Given \(x : D\), there exists some approximating family
  \(\alpha : I \to B\) for \(s(x)\). We claim that \(\alpha\) is an
  approximating family for \(x\) as well, i.e.\
  \begin{enumerate}[(i)]
  \item \(r\paren*{\beta\paren*{\alpha_i}} \ll x\) for every \(i : I\) and
  \item \(\bigsqcup r \circ \beta \circ \alpha = x\).
  \end{enumerate}
  The second follows from continuity of \(r\), since:
  \(\bigsqcup r \circ \beta \circ \alpha = r\paren*{\bigsqcup \beta \circ
    \alpha} = r\paren*{s(x)} = x\).  For (i), suppose that \(i : I\) and that
  \(\gamma : J \to D\) is a directed family satisfying
  \(x \order \bigsqcup \gamma\). We must show that there exists \(j : J\) with
  \(r\paren*{\beta\paren*{\alpha_I}} \order \gamma_j\). By continuity of \(s\),
  we get \(s(x) \order \bigsqcup s \circ \gamma\). Hence, since
  \(\beta\paren*{\alpha_i} \ll s(x)\), there must exist \(j : J\) with
  \(\beta\paren*{\alpha_i} \order s\paren*{\gamma_j}\). Thus, by monotonicity of
  \(r\), we get the desired
  \(r\paren*{\beta\paren*{\alpha_i}} \order r\paren*{s\paren*{\gamma_j}} =
  \gamma_j\).
\end{proof}
We now turn to locally small dcpos, as they allow us to find canonical
approximating families, which is used in the proof of
\cref{continuous-retract-of-algebraic}.

\begin{lemma}\label{locally-small-iff-ddarrow-small}
  Let \(D\) be a \(\V \)-dcpo with basis
  \(\beta : B \to D\). The following are equivalent:
  \begin{enumerate}[(i)]
  \item \(D\) is locally small;
  \item \(\beta(b) \ll x\) has size \(\V \) for every \(x : D\)
    and \(b : B\).
  \end{enumerate}
\end{lemma}
\begin{proof}
  Recalling \cref{order-in-terms-of-way-below'}, the type \(x \order y\) is
  equivalent to \(\prod_{b : B} \paren*{\beta(b) \ll x \to \beta(b) \ll y}\) for
  every \(x,y : D\). Thus, (ii) implies (i). Conversely, assume that \(D\) is
  locally small and let \(x : D\) and \(b : B\). We claim that
  \(\beta(b) \ll x\) is equivalent to
  \(\exists_{b' : B} \paren*{b \ll^B b' \times \beta(b') \order_{\locsmall} x} :
  \V \). The left-to-right implication is given by \cref{unary-interpolation},
  and the converse by \cref{way-below-properties}(iii).
\end{proof}

\begin{lemma}\label{canonical-approximating-family}
  Let \(D\) be a \(\mathcal V\)-dcpo with basis \(\beta : B \to D\). If \(D\) is
  locally small, then an element \(x : D\) is the supremum of the large directed
  family
  \(\paren*{\sum_{b : B} \beta(b) \ll x} \xrightarrow{\fst} B
  \xrightarrow{\beta} D\).  Moreover, if \(D\) is locally small, then this
  directed family is small.
\end{lemma}
\begin{proof}
  The family \(\fst \circ \beta\) is directed by the nullary
  (\cref{nullary-interpolation}) and binary (\cref{binary-interpolation})
  interpolation properties. Now suppose that \(D\) is locally small. By
  \cref{locally-small-iff-ddarrow-small}, we have \(I : \V\) and
  \(\alpha : I \to D\) directed such that \(\bigsqcup \alpha\) is the supremum
  of
  \(\paren*{\sum_{b : B} \beta(b) \ll x} \xrightarrow{\fst} B
  \xrightarrow{\beta} D\). Since \(\beta : {B \to D}\) is a basis of \(D\), we
  see that \(x \order \bigsqcup \alpha\). For the reverse inequality, it
  suffices to show that \(\beta(b) \order x\) for every \(b : B\) with
  \(\beta(b) \ll x\). But this follows from \cref{way-below-properties}(ii).
\end{proof}

\begin{theorem}\label{continuous-retract-of-algebraic}
  Let \(D\) be a \(\V \)-dcpo with basis
  \(\beta : B \to D\) and suppose that \(D\) is locally
  small. Then \(D\) is a continuous retract of the algebraic
  \(\V \)-dcpo \(\Idl\paren*{B,\order^B}\) (recall
  \cref{order-small-on-basis}).
\end{theorem}
\begin{proof}
  By \cref{reflexive-implies-algebraic}, \(\Idl\paren*{B,\order^B}\) is indeed
  algebraic. Let \(D\) be a \(\V \)-dcpo satisfying the hypotheses of the
  lemma. Let \({\ll_{\locsmall}} : {B \to D \to \V }\) be such that
  \(\paren*{b \ll_{\locsmall} x} \simeq \paren*{\beta(b) \ll x}\) for every
  \(x : D\) and \(b : B\).

  For every \(x : D\), we can consider the subset \(\ddarrow{x}\) given by
  \(\lambda (b : B) . b \ll_{\locsmall} x\). We show that it is an ideal.
  By \cref{canonical-approximating-family} it is a directed subset.
  And if \(b \in \ddarrow{x}\) and \(b' \order^B b\), then
  \(b' \in \ddarrow{x}\) as well by virtue of
  \cref{way-below-properties}(iii). So \(\ddarrow{x}\) is a lower set, and
  indeed an ideal.

  We claim that the map \(\ddarrow(-)\) is continuous. By
  \cref{order-in-terms-of-way-below}, it is monotone. Thus, we are left to show
  that if \(\alpha : I \to D\) is directed, then
  \(\ddarrow(\bigsqcup \alpha) \subseteq \bigsqcup_{i : I}
  \ddarrow{\alpha_i}\). Let \(b \in \ddarrow(\bigsqcup \alpha)\), i.e.\
  \(b \in B\) such that \(b \ll_\locsmall \bigsqcup \alpha\). By
  \cref{unary-interpolation}, there exists \(b' : B\) with
  \(b \ll^B b' \ll_\locsmall \bigsqcup \alpha\). Hence, there must exist
  \(i : I\) such that \(\beta(b) \ll \beta(b') \order \alpha_i\), thus,
  \(b \in \ddarrow{\alpha_i}\) and \(\ddarrow(-)\) is indeed continuous.

  Next, define \(r : \Idl\paren*{B,\order^B} \to D\) using
  \cref{free-dcpo-on-a-poset} as the unique continuous function such that
  \[
    \begin{tikzcd}
      B \ar[dr, "\principalideal(-)"'] \ar[rr, "\beta"] & & D \\
      & \Idl\paren*{B,\order^B} \ar[ur,"r"']
    \end{tikzcd}
  \]
  commutes, i.e.\ \(r\) maps an ideal \(I\) to the directed supremum
  \(\bigsqcup_{b \in I} \beta(b)\) in \(D\).

  Finally, we show that \(\ddarrow(-)\) is a section of \(r\). That is, the
  equality \(\bigsqcup_{b \ll_\locsmall x} \beta(b) = x\) holds for every
  \(x : D\). But this is exactly \cref{canonical-approximating-family}.
\end{proof}
One may wonder how restrictive the condition that \(D\) is locally
small is. We note that if \(X\) is a set, then \(\lifting_{\V }(X)\)
(by \cref{lifting-order-alt}) and \(\powerset_{\V }(X)\) are examples
of locally small \(\V \)-dcpos. A~natural question is what happens
with exponentials. In general, \(E^{D}\) may fail to be locally small
even when both \(D\) and \(E\) are. However, we do have the following
result.

\begin{lemma}\label{exponential-locally-small-criterion}
  Let \(D\) and \(E\) be \(\V \)-dcpos. Suppose that
  \(D\) is continuous and \(E\) is locally small. Then
  \(E^{D}\) is locally small.
\end{lemma}
  \begin{proof}
  Since being locally small is a proposition, we may assume that we are given a
  basis \(\beta : B \to D\) of \(D\). We claim that for every
  two continuous functions \(f,g : D \to E\) we have an
  equivalence
  \[
    \paren*{\prod_{x : D}f(x) \order_{E} g(x)} \simeq \paren*{\prod_{b :
        B}f(\beta(b)) \order_{\locsmall} g(\beta(b))}.
  \]
  Since \(B : \V \) and \(\order_{\locsmall}\) takes values in \(\V \), the
  second type is also in \(\V \). For the equivalence, note that the
  left-to-right implication is trivial. For the converse, assume the right-hand
  side and let \(x : D\). By continuity of \(D\), there exists some
  approximating family \(\alpha : I \to B\) for \(x\). We use it as follows:
  \begin{align*}
    f(x) &= f\paren*{\bigsqcup \beta \circ \alpha} \\
         &= \bigsqcup_{i : I} f\paren*{\beta\paren*{\alpha_i}}
         &&\text{(by continuity of \(f\))} \\
         &\order \bigsqcup_{i : I} g\paren*{\beta\paren*{\alpha_i}}
         &&\text{(by assumption)} \\
         &= g\paren*{\bigsqcup \beta \circ \alpha}
         &&\text{(by continuity of \(g\))} \\
         &= g(x),
  \end{align*}
  which finishes the proof.
\end{proof}
Moreover, the (co)limit of locally small dcpos is locally small.
\begin{lemma}\label{D-infty-locally-small-criterion}
  Given a system \((D_i,\varepsilon_{i,j},\pi_{i,j})\) as in
  \cref{limits-and-colimits}, if every \(D_i\) is locally small, then so is
  \(D_\infty\).
\end{lemma}
Finally, the requirement that \(D\) is locally small is necessary, in the
following sense.
\begin{lemma}
  Let \(D\) be a \(\V \)-dcpo with basis \(\beta : B \to D\). Suppose that \(D\)
  is a continuous retract of \(\Idl\paren*{B,\order^B}\). Then \(D\) is locally
  small.
\end{lemma}
\begin{proof}
  Let \(s : D \to \Idl\paren*{B,\order^B}\) be a section of a map
  \(r : \Idl\paren*{B,\order^B} \to D\), with both maps
  continuous. Then \(x \order_{D} y\) holds if and only if
  \(s(x) \order_{\Idl\paren*{B,\order^B}} s(y)\). Since
  \(\Idl\paren*{B,\order^B}\) is locally small, so must \(D\).
\end{proof}
We have now developed the theory sufficiently to give a proof of
\cref{D-infty-is-algebraic}.
\begin{claimproof}[Proof of \cref{D-infty-is-algebraic}
  (Scott's \(D_\infty\) is algebraic)]
  \label{proof-of-D-infty-is-algebraic}
    Firstly, notice that \(D_0\) is not just a \(\U_0\)-dcpo, but in fact a
    \(\U_0\)-sup lattice, i.e.\ it has joins for all families indexed by types
    in \(\U_0\). Moreover, since joins in exponentials are given pointwise,
    every \(D_n\) is in fact a \(\U_0\)-sup lattice. In~particular, every
    \(D_n\) has all finite joins. Hence, if we have \(\alpha : I \to D_n\) with
    \(I : \U_0\), then we can consider the directed family
    \(\overline{\alpha} : \overline{I} \to D_n\) with
    \(\overline{I} \colonequiv \sum_{k : \natt}\paren*{\Fin k \to D_n}\) and
    \(\overline{\alpha}\) mapping a pair \((k,f)\) to the finite join
    \(\bigvee_{0 \leq i < k} f(i)\). Moreover, if every \(\alpha_i\) is compact,
    then so is every \(\overline{\alpha}_{\overline{i}}\), since finite joins of
    compact elements are compact again. We show this explicitly for binary joins
    from which the general case follows by induction. If \(a,b : D_n\) are
    compact and \({a,b} \order {a \vee b} \order \bigsqcup \gamma\) for some
    directed family \(\gamma : {J \to D_n}\), then by compactness of \(a\) and
    \(b\), there exist \({j_a,j_b} : J\) such that \(a \order \gamma_{j_a}\) and
    \(b \order \gamma_{j_b}\). By directedness of \(\gamma\), there exists
    \(k : J\) with \({a,b} \order \gamma_k\). Hence,
    \(a \lor b \order \gamma_k\), as desired.
  \begin{claim*}
    Every \(D_n\) is locally small and has a basis \(\beta_n : B_n \to D_n\) of
    compact elements.
  \end{claim*}
  \begin{claimproof}
    We prove this by induction. For \(n = 0\), this follows from
    \Cref{lifting-order-alt,lifting-is-algebraic}. Now suppose that \(B_m\) is
    locally small and has a basis \(\beta_m : {B_m \to D_m}\). By
    \cref{exponential-locally-small-criterion}, the dcpo
    \(B_{m+1} \equiv B_m^{B_m}\) is locally small. If we have \({a,b} : B_m\),
    then we define the continuous \emph{step~function}
    \(\paren*{a \Rightarrow b} : {D_m \to D_m}\) by
    \(x \mapsto \bigvee_{{\beta_m\paren*{a}} \order {x}} \beta_m(b)\), which is
    well-defined since \(D_m\) is locally small. We are going to show that
    \(a \Rightarrow b\) is compact for every \(a,b : B_m\) and that every
    \(f : D_{m+1}\) is the join of certain step functions.
    To this end, we first observe that
    \begin{equation*}
      \tag{\(\dagger\)}
      \label{step-function-below}
      \paren*{{a \Rightarrow b} \order f} \iff
      \paren*{\beta_m(b) \order f\paren*{\beta_m(a)}},
    \end{equation*}
    which follows from the fact that continuous functions are monotone.

    For compactness, suppose that
    \(a \Rightarrow b \order \bigsqcup_{i : I} f_i\). By
    \eqref{step-function-below} we have
    \(\beta_m(b) \order \bigvee_{i :I} \paren*{f_i\paren*{\beta_m(a)}}\). By
    compactness of \(\beta_m(b)\), there exists \(i : I\) such that
    \(\beta_m(b) \order f_i\paren*{\beta_m(a)}\) already. Using
    \eqref{step-function-below} once more, we get the desired
    \(a \Rightarrow b \order f_i\).

    Now let \(f : D_m \to D_m\) be continuous. We claim that \(f\) is the join
    of the step-functions below it, i.e.\
    \( f = \bigvee_{{{a , b} : B_m , {a \Rightarrow b} \order f}} a \Rightarrow
    b \), which is well-defined, since \(D_{m+1}\) is locally small. One
    inequality clearly holds as we are only considering step-functions below
    \(f\). For the reverse inequality, let \(x : D_m\) be arbitrary.
    By
    \cref{canonical-approximating-family}, we have:
    \begin{equation*}
      \tag{\(\ddagger\)}
      \label{approximating-families}
      x = \bigsqcup_{\substack{a' : B_m \\ \beta_m(a') \ll x}}\beta_m(a') \quad
      \text{and}\quad
      f(x) = \bigsqcup_{\substack{b' : B_m \\ \beta_m(b') \ll f(x)}} \beta_m(b').
    \end{equation*}
    Hence, it suffices to show that
    \(\beta_m(b') \order \bigvee_{a,b : B_m , {a \Rightarrow b} \order f}
    \paren*{a \Rightarrow b}(x)\) whenever \(b' : B_m\) is such that
    \(\beta_m(b') \ll f(x)\). By \eqref{step-function-below} and the definition
    of a step-function it is enough to find \(a' : B_m\) such that
    \(\beta_m(b') \order f\paren*{\beta_m(a')}\) and \(\beta_m(a') \order
    x\). Using \eqref{approximating-families}, our assumption
    \(\beta_m(b') \ll f(x)\) and continuity of \(f\), we get that there exists
    \(a' : B_m\) with \(\beta_m(a') \ll x\) (and thus \(\beta_m(a') \order x\))
    and \(b_m(b') \order f\paren*{\beta_m(a')}\), as desired.

    Thus, by the paragraph preceding the claim, \(D_{m+1}\) has a basis of
    compact elements:
    \(\beta_{m+1} : {\paren*{\sum_{k : \natt}\paren*{\Fin(k) \to \paren*{D_m \times
          D_m}}} \to D_{m+1}}\) with
    \(\beta_{m+1}\paren*{k,\lambda i . \paren*{a_i,b_i}} \colonequiv \bigvee_{0
      \leq i < k} a_i \Rightarrow b_i\), finishing the proof of the claim.
  \end{claimproof}

  Finally, we show that a basis of compact elements for \(D_\infty\) is
  \(\beta_\infty : {\paren*{B_\infty \colonequiv \sum_{n : \natt} B_n} \to
    D_\infty}\) where
  \(\beta_\infty (n , b) \colonequiv
  \varepsilon_{n,\infty}\paren*{\beta_n(b)}\). We first check compactness by
  showing that if \(x : D_n\) is compact, then so is
  \(\varepsilon_{n,\infty}(x)\). This follows easily from the fact that
  \(\paren*{\varepsilon_{n,\infty},\pi_{n,\infty}}\) is an
  embedding-projection. For if \(\alpha : I \to D_\infty\) is directed and
  \(\varepsilon_{n,\infty}(x) \order \bigsqcup \alpha\), then
  \( {x = \pi_{n,\infty}\paren*{\varepsilon_{n,\infty}\paren*{x}} \order
  \pi_{n,\infty}\paren*{\bigsqcup \alpha} = \bigsqcup \pi_{n,\infty} \circ
  \alpha}, \) by~continuity of~\(\pi_{n,\infty}\). Thus, by compactness of \(x\),
  there exist \(i : I\) such that \(x \order \pi_{n,\infty}\paren*{\alpha_i}\)
  already. Hence,
  \(\varepsilon_{n,\infty}(x) \order
  \varepsilon_{n,\infty}\paren*{\pi_{n,\infty}\paren*{\alpha_i}} \order
  \alpha_i\), so \(\varepsilon_{n,\infty}(x)\) is indeed compact. Now let
  \(\sigma : D_\infty\) be arbitrary. As mentioned in the proof of
  \cref{colimit}, we have
  \(\sigma = \bigsqcup_{n : \natt} \varepsilon_{n,\infty}(\sigma_n)\). By
  \cref{canonical-approximating-family} and the claim, we can express every
  \(\sigma_n : D_n\) as
  \(\bigsqcup_{b : B_n , \beta_n(b) \ll \sigma_n} \beta_n(b)\).  Hence,
  \[
    \sigma = \bigsqcup_{n : \natt} \varepsilon_{n,\infty}
    \paren*{\bigsqcup_{\substack{b : B_n \\ \beta_n(b) \ll \sigma_n}} \beta_n(b)}
    = \bigsqcup_{n : \natt}
    \bigsqcup_{\substack{b : B_n \\ \beta_n(b) \ll \sigma_n}}
    \varepsilon_{n,\infty}\paren*{\beta_n(b)}
  \]
  by continuity of \(\varepsilon_{n,\infty}\). Thus, \(\sigma\) may be expressed
  as the supremum of the directed family
  \(\paren*{\sum_{n : \natt}\sum_{b : B_n}\beta_n(b) \ll \sigma_n}
  \to B_\infty \xrightarrow{\beta_\infty} D_\infty\).  (And in
  light of \cref{locally-small-iff-ddarrow-small} and the claim, the type
  \(\sum_{n : \natt}\sum_{b : B_n}\beta_n(b) \ll \sigma_n\) can be replaced by a
  type in \(\U_0\).) Finally, using \cref{algebraic-criterion}, we see that
  \(D_\infty\) is indeed algebraic.
\end{claimproof}
We end this section by describing an example of a continuous dcpo, built using
the ideal completion, that is not algebraic. In fact, this dcpo has no compact
elements at all.
\begin{example}[A continuous dcpo that is not algebraic]
  \label{ideal-completion-of-dyadics}

  We inductively define a type and an order representing dyadic rationals
    \(m / 2^n\) in the interval \((-1,1)\) for integers \(m,n\). The~intuition
    for the upcoming definitions is the following. Start with the point \(0\) in
    the middle of the interval (represented by \(\dyadiccenter\) below). Then
    consider the two functions (respectively represented by \(\dyadicleft\) and
    \(\dyadicright\) below)
    \begin{align*}
      l,r &: (-1,1) \to (-1,1) \\
      l(x) &= (x-1)/2 \\
      r(x) &= (x+1)/2
    \end{align*}
    that generate the dyadic rationals. Observe that \(l(x) < 0 < r(x)\) for
    every \(x : (-1,1)\). Accordingly, we inductively define the following
    types.

  \begin{definition}[{Dyadics} \(\mathbb D\)]
    The type of \emph{dyadics} \(\mathbb D : \U_0\) is the inductive type with
    three constructors:
    \[
      \dyadiccenter : \mathbb D \quad \dyadicleft : {\mathbb D \to \mathbb D}
      \quad \dyadicright : {\mathbb D \to \mathbb D}.
    \]
  \end{definition}
  \begin{definition}[Order \(\prec\) on \(\mathbb D\)]
    Let \({\prec} : {\mathbb D \to \mathbb D \to \U_0}\) be inductively defined
    as:
    \begin{alignat*}{6}
      \dyadiccenter & \prec \dyadiccenter && \colonequiv \emptyt &\qquad\qquad
      {\dyadicleft x} & \prec {\dyadiccenter} && \colonequiv \unitt
      &\qquad\qquad
      {\dyadicright x} & \prec {\dyadiccenter} && \colonequiv \emptyt \\
      \dyadiccenter & \prec {\dyadicleft y} && \colonequiv \emptyt &
      {\dyadicleft x} & \prec {\dyadicleft y} && \colonequiv {x \prec y} &
      {\dyadicright x} & \prec {\dyadicleft y} && \colonequiv \emptyt \\
      \dyadiccenter & \prec {\dyadicright y} && \colonequiv \unitt &
      {\dyadicleft x} & \prec {\dyadicright y} && \colonequiv \unitt &
      {\dyadicright x} & \prec {\dyadicright y} && \colonequiv {x \prec y}.
    \end{alignat*}
  \end{definition}
  One then shows that \(\prec\) is proposition-valued, transitive, irreflexive,
  trichotomous, dense and that it has no endpoints. \emph{Trichotomy} means that
  exactly one of \(x \prec y\), \(x = y\), \(y \prec x\) holds. \emph{Density}
  says that for every \(x,y : \mathbb D\), there exists some \(z : \mathbb D\)
  such that \(x \prec z \prec y\). Finally, \emph{having no endpoints} means
  that for every \(x : \mathbb D\), there exist some \(y,z : \mathbb D\) with
  \(y \prec x \prec z\).
  \Omit{
  We can inductively define a type \(\mathbb D\) and an order
  \(\prec\) representing dyadic rationals \(m / 2^n\)
  in the interval \((-1,1)\) for integers \(m,n\). Then we prove that \(\prec\) is
  proposition-valued, transitive, irreflexive, trichotomous, dense and
  that it has no endpoints.
  }
  Using these properties, we can show that
  \((\mathbb D,\prec)\) is a \(\U_0\)-abstract basis. Thus, taking the
  rounded ideal completion, we get
  \(\Idl(\mathbb D,\prec) : \DCPOnum{0}{1}{0}\), which is continuous
  with basis
  \(\principalideal(-) : \mathbb D \to \Idl(\mathbb D,\prec)\) by
  \cref{Idl-is-continuous}.  But \(\Idl(\mathbb D,\prec)\) cannot be
  algebraic, since none of its elements are compact. Indeed suppose
  that we had an ideal \(I\) with \(I \ll I\). By
  \cref{ideals-way-below-in-terms-of-containment}, there would exist
  \(x \in I\) with \(I \order \principalideal{x}\). But this implies
  \(x \prec x\), but \(\prec\) is irreflexive, so this is
  impossible. 
\end{example}

\section{Conclusion and Future Work} \label{conclusion}

We have developed domain theory constructively and predicatively in univalent
foundations, including Scott's \(D_\infty\) model of the untyped
\(\lambda\)-calculus, as well as notions of continuous and algebraic dcpos. We
avoid size issues in our predicative setting by having large dcpos with joins of
small directed families. Often we find it convenient to work with locally small
dcpos, whose orders have small truth values.

In future work, we wish to give a predicative account of the theory of algebraic
and continuous exponentials, which is a rich and challenging topic even
classically. We also intend to develop applications to topology and locale
theory. First steps on formal topology and frames in cubical type
theory~\cite{CohenEtAl2018,CoquandHuberMortberg2018} are developed
in Tosun's thesis~\cite{Tosun2020}, and our notion of continuous dcpo should be
applicable to tackle local compactness and exponentiability.

It is also important to understand when
classical theorems do not have constructive and predicative
counterparts. For instance, Zorn's Lemma doesn't imply excluded middle
but it implies propositional resizing \cite{deJong2020} and we are
working on additional examples.

We have formalized the following in Agda~\cite{TypeTopologyFork}, in addition to
the Scott model of PCF and its computational
adequacy~\cite{deJong2019,Hart2020}:
\begin{enumerate}
\item dcpos,
\item limits and colimits of dcpos, Scott's \(D_\infty\),
\item lifting and exponential constructions,
\item pointed dcpos have subsingleton joins (in the right universe),
\item way-below relation, continuous, algebraic dcpos, interpolation properties,
\item abstract bases and rounded ideal completions (including its universal property),
\item continuous dcpos are continuous retract of their ideal completion, and hence of algebraic dcpos,
\item ideal completion of dyadics, giving an example of a non-algebraic, continuous dcpo.
\end{enumerate}
In the near future we intend to complete our formalization to also include
\Cref{lifting-is-free,lifting-is-free2,lifting-is-free3,powerset-is-algebraic,D-infty-is-algebraic,exponential-locally-small-criterion}.

\bibliography{csl}





\end{document}